\documentclass[12pt]{amsart}
\usepackage{amsmath}
\usepackage{amsthm}
\usepackage{amsfonts}
\usepackage{amssymb}
\usepackage{verbatim}
\usepackage{graphicx,xcolor}
\usepackage[margin=1in]{geometry}
\usepackage{float}
\usepackage{eucal}
\usepackage{newtxmath}
\newcommand{\beq}{\begin{equation}}
\newcommand{\eeq}{\end{equation}}
\swapnumbers
\numberwithin{equation}{section}

\newtheorem{theorem}{Theorem}[section]

\newtheorem{definition}[theorem]{Definition}

\newtheorem{lemma}[theorem]{Lemma}

\newtheorem{proposition}[theorem]{Proposition}

\newtheorem{remark}[theorem]{Remark}

\newcommand{\R}[1]{\mathbb{R}^{#1}}

\newcommand{\C}[1]{\mathbb{C}^{#1}}

\newcommand{\1}{\vmathbb{1}}
\newcommand{\eps}{\varepsilon}
\newcommand{\abs}[1]{\left\vert{#1}\right\vert}

\begin{document}
\title{Refined Asymptotics for Landau-de Gennes Minimizers on Planar Domains}

\author{Dmitry Golovaty}
\address{Department of Mathematics, The University of Akron, Akron, OH 44325, USA}
\email{dmitry@uakron.edu}

\author{Jose Alberto Montero}
\email{j.alberto.montero.z@gmail.com}

\begin{abstract}
In our previous work \cite{GM}, we studied asymptotic behavior of minimizers of the Landau-de Gennes energy functional on planar domains as the nematic correlation length converges to zero. Here we improve upon those results, in particular by sharpening the description of the limiting map of the minimizers. We also provide an expression for the energy valid for a small, but fixed value of the nematic correlation length.  
\end{abstract}

\maketitle

In this paper we revisit some of the conclusions we obtained in \cite{GM}. In that paper we considered the Landau-de Gennes energy functional, which can be expressed as
\begin{equation}\label{LdGEnergy}
E_\eps(u) = \int_\Omega \left ( \frac{\abs{\nabla u}^2}{2} + \frac{W_\beta(u)}{\eps^2}\right).
\end{equation}
Here $\Omega \subset \R{2}$ is a bounded, smooth, simply-connected open subset of the plane and $\eps > 0$ is a small parameter known as the {\it nematic correlation length}. In \cite{GM} we considered the functional $E_\eps$ among maps $u \in W^{1, 2}(\Omega, M^3_{s, 1}(\R{}))$, where $M^3_{s, 1}(\R{})$ denotes the set of symmetric, $3\times 3$ matrices with real entries and trace equal to $1$.  The potential $W_\beta$ can be expressed as
$$
W_\beta(u) = \frac{(1-\abs{u}^2)^2}{4}-\beta\rm{det}(u),
$$
where $1\leq \beta < 3$; here and throughout the paper, for two matrices $A, B$ of the same size, we consider the inner product $\langle A, B\rangle = {\rm tr}(B^TA)$, along with its induced norm $\abs{A}^2 =\langle A, A\rangle$.  For $\beta \in [1, 3[$ the potential $W_\beta$ is minimized \cite{GM} by the elements of the set
$$
\mathcal{P} = \{P \in M^3_{s, 1} : \,\,  P^2=P\}
$$
of $3\times 3$, rank-1, orthogonal projection matrices.  Our aim in \cite{GM} was to study the global minimizers of $E_\eps$, in the limit $\eps \to 0$, among maps $u \in W^{1, 2}(\Omega, M^3_{s, 1}(\R{}))$ that satisfy the boundary condition $u = u_b$ on $\partial \Omega$.  A crucial hypothesis in \cite{GM} was that
$$
u_b : \partial \Omega \to \mathcal{P}
$$
represents a non-contractible curve in $\mathcal P$.

Roughly speaking, if $u_\eps$ denotes a global minimizer of $E_\eps$ under the conditions we just described, the results of \cite{GM} show that, along subsequences denoted by $\eps_n \to 0$, there is a single interior point $a \in \Omega$ such that $u_{\eps_n}\to u$ in $W^{1, 2}(\Omega \setminus B_r(a))$ for any fixed $r > 0$, where $u : \Omega \to \mathcal{P}$ is a projection-valued map.  Furthermore, the limit map $u$ locally minimizes the Dirichlet integral in $\Omega\setminus \{a\}$, so it is a harmonic map.  Finally, if we write $\frac{\partial}{\partial z}$ for the standard complex derivative on the plane, and $[A, B]=AB-BA$ denotes the commutator of the matrices $A$ and $B$, we showed that the {\it current vector} of $u$, defined by
$$
j(u) = \left [  u, \frac{\partial u}{\partial z} \right ],
$$
splits as a sum of a meromorphic function with an explicit singular part, plus a map in $W^{1, 2}$.  The appendix contains a more detailed description of the results from \cite{GM} and elsewhere that will be relevant to the present work.

As in \cite{GM}, in this paper we consider the minimizers of \eqref{LdGEnergy} among all maps $u\in W^{1, 2}(\Omega, M^3_{s, 1})$ which satisfy the condition $u = u_b$ on $\partial \Omega$, where
$$
u_b : \partial \Omega \to \mathcal{P}
$$
is a fixed, non-contractible curve in $\mathcal P$. We improve upon the results from \cite{GM} in two ways.  First, we find a more detailed description for the limits $u$ of global minimizers $u_\eps$ of $E_\eps$.  Second, we use this refined description to provide an expansion of the energy $E_\eps(u_\eps)$ of global minimizers, valid for small, but fixed, $\eps > 0$. 

\medskip
\medskip

The first result giving a better description of the limits of global minimizers of $E_\eps$ as $\eps \to 0$ is summarized in the following proposition (here and elsewhere in the paper $\mathbb{S}^k$ denotes the unit sphere in $\R{k+1}$):

\begin{proposition}\label{convergence_geodesic_third_eval}
Let $a\in \Omega$ be the distinguished point given by Theorem 1 of \cite{GM}, $\eps_n \to 0$, and $u_{\eps_n}\in W^{1, 2}(\Omega, M^3_{s, 1})$ be a sequence of global minimizers of $E_{\eps_n}$ such that $u_{\eps_n}\to u$ in $W^{1, 2}_{loc}(\Omega \setminus \{a\}, M^3_{s, 1})$.  There is a unit vector-valued map $k \in W^{1, 2}(\Omega, \mathbb{S}^2)$ such that $u(x) k(x) = 0$ for all $x\in \Omega \setminus \{a\}$. 

\medskip
\medskip

\noindent Furthermore, if we define $\gamma_r : \mathbb{S}^1\to \mathcal P$ by
$$
\gamma_r(\omega) = u(a+r\,\omega),
$$
then there is a closed geodesic $\gamma_0 : \mathbb{S}^1\to {\mathcal P}$ such that $\gamma_r \to \gamma_0$ as $r\to 0$ in $W^{1, 2}(\mathbb{S}^1, {\mathcal P})$ 
\end{proposition}
\begin{remark}
We emphasize that in the above proposition the convergence $\gamma_r \to \gamma_0$ in $W^{1, 2}(\mathbb{S}^1, {\mathcal P})$ is for $r \to 0$, not along a particular sequence $r_n \to 0$.  Results in this spirit appear in \cite{GM} and \cite{MRvS02}, but only for sequences $r_n \to 0$.

It is also worth mentioning that this proposition confirms the intuition that, while the first two eigenvectors of the map $u$ are singular at $a\in \Omega$, the third eigenvector should be smooth---although at this point we can only prove that it is in $W^{1, 2}$.
\end{remark}

Our next result makes use of the Hopf differential of the map $u$, the definition of which we recall in \eqref{def_hopf_diff}.

\begin{theorem}\label{current_vector}
Let $\omega_u$ denote the Hopf differential of the map $u$, which we assume to be a limit of global minimizers of $E_\eps$.  From the Proposition \ref{Hopf_Diff_split} we have
$$
\omega_u(z) = -\frac{1}{8(z-a)^2} + h(z),
$$
where $h$ is a holomorphic map in all of $\Omega$.
\medskip

\noindent Let now $Z_{\omega_u}$ denote the set of zeros of $\omega_u$ in $\Omega$.  Under the hypothesis that $Z_{\omega_u} = \emptyset$, there are
\begin{enumerate}
\item a fixed orthogonal basis $\Lambda_1$, $\Lambda_2$, $\Lambda_3$ of the set $M^3_a(\R{})$ of $3 \times 3$ anti-symmetric matrices that satisfies
$$
[\Lambda_1, \Lambda_2] = \Lambda_3, \,\,\,\,  \Lambda_j^3 = - \Lambda_j, \,\,\, j=1, 2, 3,
$$
\item a real-valued function $g: \Omega \setminus \{a\} \to \R{},$ defined up to a sign,
\item a fixed projection $P \in {\mathcal P}$ and 
\item a multi-valued map $S : \Omega \setminus \{a\} \to O(3)$
\end{enumerate}
such that
$$
u = SPS^T \,\,\,\,\mbox{and}\,\,\,\,  j(u) = \mu_u S( \cosh(g) \Lambda_1 + i \sinh(g) \Lambda_2) S^T,
$$
where $-2\mu_u^2 = \omega_u$.  Furthermore, letting $\Gamma_j = S\Lambda_jS^T$, $j=1, 2, 3$, the function $g$ satisfies
$$
-\frac{i}{4} (\Delta g)\, \Gamma_3 = \frac{i\abs{\omega_u}}{4} \sinh(2g) \Gamma_3 = \frac{1}{2}[\overline{j(u)}, j(u)] \,\,\,\,\mbox{locally in}\,\,\,\, \Omega.
$$
When $Z_{\omega_u}\neq \emptyset$, it is discrete in $\Omega$ and all conclusions of the theorem remain valid locally away from $Z_{\omega_u}$.  This is due to the fact that the equation $-2\mu_u^2 = \omega_u$ is no longer valid globally in $\Omega \setminus \{a\}$.
\medskip

\noindent Finally, regardless of the nature of $Z_{\omega_u}$, we also have
$$
\int_\Omega \abs{\omega_u}\sinh^2(g) < +\infty.
$$

\end{theorem}
\noindent Before we state our next results, several comments are in order.  First, for the map $S : \Omega \setminus \{a\} \to O(3)$ to be multivalued we must have
$$
S(r, \theta + 2\pi) \neq S(r, \theta),
$$
using the polar coordinates centered at $a\in \Omega$. However, if we let
$$
O_P(3) = \{R\in O(3) : RPR^T = P\}
$$
to be the stabilizer in $O(3)$ of the projection $P \in \mathcal{P}$ singled out in the last theorem, then
$$
S(r, \theta + 2\pi)^TS(r, \theta) \in O_P(3).
$$
One can think of the map $S$ as a lift of $u$ through $O(3)$.  Using the properties of $\mathcal{P}$-valued harmonic maps, we prove that
$$
\frac{\partial S^T}{\partial z}S = -i\frac{\partial g}{\partial z}\Lambda_3 + \mu_u ( \cosh(g) \Lambda_1 + i \sinh(g) \Lambda_2),
$$
where $g$ and $\Lambda_j$, $j=1, 2, 3$ are as in Theorem \ref{current_vector}.  This last equation can be thought of as a set of over-determined differential equations satisfied by $S$.  Then, the equation
$$
-\Delta g\, = \abs{\omega_u} \sinh(2g)
$$
is the compatibility condition for the over-determined equations satisfied by $S$.  

Now looking at the equations satisfied by $g$ stated in Theorem \ref{current_vector}, we notice in particular that $g = 0$ identically if and only if $\left [ \frac{\partial u}{\partial \bar{z}},  \frac{\partial u}{\partial z}  \right ]=-[\overline{j(u)}, j(u)] = 0$ in all of $\Omega$.  Thinking of $u:\Omega \setminus \{a\} \to \mathcal{P}$, locally, as a parameterization of a portion of $\mathcal{P}$, then $\abs{[\overline{j(u)}, j(u)]}$ is the area factor of this parameterization.  This says that $g$ vanishes identically if and only if the image of $u$ has zero $2$-dimensional area.  Our results show that in this case the image of $u$ is contained in a closed geodesic of $\mathcal{P}$, and the map $u$ has the same structure that the canonical harmonic maps of \cite{BBH}.

Finally, the properties of the Hopf differential, particularly Proposition \ref{Hopf_Diff_split}, was pointed out in \cite{MRvS01}.  For the sake of completeness, we provide a proof of this, which follows closely that of \cite{BBH}.

\medskip
\medskip

Having formulated the results concerning the limit map $u$, we state now an expansion of the energy valid for a family $u_{\eps_n}$ of converging global minimizers of $E_{\eps_n}$.

\begin{theorem}\label{main_theo}
Let $\eps_n \to 0$, let $u_n\in W^{1, 2}(\Omega, M_{s, 1}^3(\R{}))$ be a global minimizer of $E_{\eps_n}$, and assume $u_n \to u$ in $W^{1,2}(\Omega \setminus B_r(a), M^3_{s, 1}(\R{}))$ for every fixed $r>0$.  Let also $\omega_u$ denote the Hopf differential of $u$, defined in Definition \ref{def_hopf_diff}, and let $g:\Omega \setminus \{a\} \to \R{}$ be the multi-valued function described in Theorem \ref{current_vector}.  We have the expansion
\begin{align}
\int_{\Omega}e_{\eps_n}(u_n) = I(r, \eps_n) + 2\int_{\Omega \setminus B_r(a)} \abs{\omega_u}  + 2\int_{\Omega \setminus B_r(a)} \abs{\omega_u}\sinh^2(g) + o(1) + q(r).\label{energy_expansion}
\end{align}
Here $I(r, \eps)$ is defined in equation \ref{energy_canonical_flat}, $o(1)$ represents a quantity that goes to $0$ as $n\to \infty$, and $q(r)$ represents a quantity that is independent of $\eps > 0$ and such that $q(r) \to 0$ as $r\to 0$.

\end{theorem}

\begin{proof}[Proof of Theorem \ref{main_theo}.]
We split
$$
\int_{\Omega}e_{\eps_n}(u_n) = \int_{B_r(a)}e_{\eps_n}(u_n) + \int_{\Omega \setminus B_r(a)}e_{\eps_n}(u_n).
$$
The estimate of the difference
$$
\int_{B_r(a)}e_{\eps_n}(u_n) - I(r, \eps_n)
$$
is contained in Theorem \ref{near_sing}.

\medskip
\medskip

Next, by the results in the appendix we have
$$
\int_{\Omega \setminus B_r(a)}e_{\eps_n}(u_n) = \int_{\Omega \setminus B_r(a)}\frac{\abs{\nabla u}^2}{2} + o(1).
$$
The results in Theorem \ref{current_vector} show that
$$
\int_{\Omega \setminus B_r(a)}\frac{\abs{\nabla u}^2}{2} = 2\int_{\Omega \setminus B_r(a)} \abs{\omega_u}  + 2\int_{\Omega \setminus B_r(a)} \abs{\omega_u}\sinh^2(g).
$$

\end{proof}

In the reference \cite{GM} we showed that the current vector of a limit map $u$ has the expression
\begin{equation}\label{current_split}
j(u) = \nabla^\perp \left ( \frac{1}{2}\ln\left ( \frac{1}{\abs{x-a}}\right ) \Lambda + \phi \right ),
\end{equation}
where $\Lambda$ is a constant, $3\times 3$ anti-symmetric matrix normalised so that $\Lambda^3 = -\Lambda$, and $\phi \in (W^{1, 2}\cap L^\infty)(\Omega, M^3_a(\R{}))$.  So far we have been unable to show that the map $\phi$ is smooth.  However, if we assume this, we can give another expression for the energy of global minimizers $u_\eps$ that converge in $W^{1,2}_{loc}(\Omega \setminus \{a\})$ to a projection-valued map $u$.  This is the content of our next theorem.  To state it, we assume the current vector of $u$ can be written in the form \eqref{current_split}, where $\phi \in W^{1, \infty}(\Omega, M^3_a(\R{}))$.  Note that by adding and subtracting the regular part of the Green's function for $\Omega$ we can write 
$$
j(u) = \nabla^\perp \left ( \pi G(x, a) \Lambda + \phi_1 \right ),
$$
where $G(x, y)$ is the Green's function for $\Omega$, and we then re-state our hypothesis as $\phi_1 \in W^{1, \infty}(\Omega)$. We now present our last theorem.
\begin{theorem}\label{expansion_psi}
Let $\eps_n \to 0$, let $u_n\in W^{1, 2}(\Omega, M_{s, 1}^3(\R{}))$ be a minimizer of $E_{\eps_n}$, and assume $u_n \to u$ in $W^{1,2}(\Omega \setminus B_r(a), M^3_{s, 1}(\R{}))$ for every fixed $r>0$.  
With the notation above we have
\begin{align*}
\int_\Omega e_{\eps_n}(u_n) &= I(r, \eps_n) + \frac{\pi}{2} \ln\left (\frac{1}{r} \right) + \frac{R(a, a)}{2} \\ &+  \int_\Omega G(x, a) \langle \Lambda , \left [ \frac{\partial u}{\partial x_1}, \frac{\partial u}{\partial x_2} \right ] \rangle \,dx + \frac{1}{2}\int_\Omega \abs{\nabla \phi_1}^2 + q(r) + o(1).
\end{align*}
As in our last Theorem, $I(r, \eps)$ is defined in equation \ref{energy_canonical_flat}, $o(1)$ represents a quantity that goes to $0$ as $n\to \infty$, and $q(r)$ represents a quantity that is independent of $\eps > 0$ and such that $q(r) \to 0$ as $r\to 0$.  Finally,
$$
R(x, y) = G(x, y) - \frac{1}{2\pi}\ln\left ( \frac{1}{\abs{x-y}}\right )
$$
is the regular part of the Green's function for $\Omega$.
\end{theorem}

\medskip
\medskip

Many authors have studied the Landau-de Gennes energy \eqref{LdGEnergy} in the last decade, particularly in the limit limit as $\eps \to 0$.  The authors of \cite{HM}, \cite{HMP} and \cite{MZ} all provide descriptions of the global minimizers of $E_\eps$ in the limit $\eps \to 0$, in a $3$-dimensional domain.  Several other problems related to \ref{LdGEnergy} in $3$-d domains have been studied, from homogenization via $\Gamma$-convergence, to stability of particular solutions, to the appearance of line defects in the minimizers in the limit $\eps \to 0$, in \cite{Can02}, \cite{CZ02}, \cite{CZ01}, \cite{INSZ01} and \cite{NZ}.  In $2$ dimensions, \cite{dFRSZ02}, \cite{dFRSZ01}, \cite{INSZ02}, \cite{INSZ03}, \cite{INSZ04}, \cite{KRSZ}, among other results, prove existence and multiplicity of symmetric solutions under appropriate boundary conditions, and study stability of point defects.

The study \cite{Can01} is perhaps the closest to the issues considered in the present work.  In \cite{Can01} the author considers a family of energy functionals that contain \ref{LdGEnergy} as a particular case, and establishes convergence of minimizers in the $\eps \to 0$ limit, among other results.

Also related to our work is that of \cite{MRvS02} and \cite{MRvS01}. There, the authors consider an energy that is significantly more general than $E_\eps$ from \ref{LdGEnergy}.  They analyze singular limit $\eps \to 0$, find a $\Gamma$-limit for this energy, and obtain an energy for the location of the singularities that appear in minimizers of $E_\eps$ as $\eps\to 0$.  Their results apply to a wide range of manifolds, that include $\mathcal{P}$ as a particular case.  Because of this generality, however, their results for the $\Gamma$-limit of $E_\eps$ are rather implicit, and can be made explicit only for very special boundary conditions.

Some of the tools used in \cite{dFRSZ02} are similar to ours, but this work deals with a completely different regime. In particular, in this paper the authors consider families of functions $u_\eps \in W^{1, 2}(\Omega, M^3_s(\R{}))$ such that $E_\eps(u_\eps) \leq C$ as $\eps \to 0$, for some constant $C>0$ independent of $\eps$.  Another important difference with our work is that, throughout \cite{dFRSZ02} the authors assume that their boundary data $u:\partial \Omega \to \mathcal{P}$ satisfies $ue=0$, where $e\in \mathbb{S}^2$ is a fixed unit vector.  They also assume that their boundary condition $u:\partial \Omega \to \mathcal{P}$ can be lifted through a smooth $n:\partial \Omega \to \mathbb{S}^2$, in the sense that
$$
u = nn^T \,\,\, \mbox{on}\,\,\,  \partial \Omega.
$$
We do not assume either of these hypotheses in our work.

Perhaps the main contributions of our paper are the proofs of the results we present here.  A first crucial fact we appeal to is that, for projection-valued maps $u:\Omega \to \mathcal{P}$, the current vector
$$
j(u) = \left [  u, \frac{\partial u}{\partial z}\right ]
$$
can be thought of as a set of differential equations satisfied by $u$.  Concretely, for the projection-valued map $u$ we have
$$
\frac{\partial u}{\partial z} = \left [  u, j(u) \right ],
$$
and this equation holds pointwise in $\Omega$.  This fact, along with the decomposition \ref{current_split} of $j(u)$ we found in \cite{GM}, and Proposition \ref{properties_psi}, which follows the arguments of \cite{BBH}, allow us to derive several of our conclusions.

A second important fact we use is that there is an integrable system that appears naturally in the study of projection valued maps that arise as limits of minimizers of $E_\eps$.  Indeed, for such a projection-valued map $u:\Omega \to \mathcal{P}$, we have
$$
\frac{\partial}{\partial \bar{z}} j(u) = -[\overline{j(u)}, j(u)].
$$
Lifting $u$ locally through a map $S:\Omega \to O(3)$, in the sense that $u=SPS^T$ for a fixed $P\in \mathcal{P}$, we can use this last equation to derive a differential system for $S$ that gives us Theorem \ref{current_vector}.  This in turn allows us to derive Theorem \ref{main_theo}.  The integrable system we obtain is well-known to geometers, and seems to have first appeared in \cite{DPW}.  To the best of our knowledge, however, it has not been used to this point in the Landau-de Gennes literature.

We believe that our methods, while currently restricted to the manifold $\mathcal{P}$, should lead to more explicit expressions, particularly for the energy for the location of singularities, than those currently available.  It is also worth mentioning that our methods seem to provide a natural generalization to the complex-valued methods used in \cite{BBH}.

In the remainder of the paper we provide the proofs of the results we have just described.  Section 2 contains the analyzis of the limiting map $u$.  In Section 3 we analyze the energy of a sequence global minimizers $u_{\eps_n}$, $\eps_n \to 0$, near the singular point $a\in \Omega$.  In Section 4 we prove Theorem \ref{current_vector}. In Section 5. we use numerical simulations to illustrate our results.  Finally, the Appendix contains details of results from \cite{GM} and elsewhere that we need in this work.

\section{The limiting map}

Throughout this section $u:\Omega \setminus \{a\} \to \mathcal P$ will denote a limit of minimizers of the Landau-de Gennes energy, and $a\in \Omega$ will be the unique singularity of $u$ in $\Omega$.  We will assume throughout that $a=0 \in \R{2} \cong \C{}$.  Let us recall here that $u$ satisfies
$$
[u, \Delta u]=0.
$$
We will also use the notation
$$
j(u) = [u, \nabla u]
$$
for the current vector.  Note for future reference that for any projection-valued map we have the identity
$$
\nabla u = [u, j(u)].
$$
We interpret this expression as saying that the current vector $j(u)$ contains the coefficients of the differential equation $\nabla u = [u, j(u)]$ satisfied by $u$.

\medskip
\medskip

We know from \cite{GM} that
$$
j(u) = \frac{\hat{\theta}}{2r} \Lambda + \nabla^\perp \phi.
$$
Here $r = \abs{x-a}$ is the distance to the singularity, $\hat{\theta}$ is the standard unit vector from polar coordinates centered at $a = 0 \in \R{2}$, $\phi \in (L^\infty \cap W^{1, 2})(\Omega, M_a^3(\R{}))$, and $\Lambda \in M^3_a(\R{})$ is a constant anti-symmetric matrix, the representative of a closed geodesic in $\mathcal P$ in the language of \cite{GM}.  By choosing appropriate coordinates in $\mathcal P$ we can choose
$$
\Lambda = \left (  \begin{array}{ccc} 0 & -1 & 0 \\ 1 & 0 & 0 \\ 0 & 0 & 0  \end{array} \right ).
$$

Let us consider now the domain
$$
U = \{\xi \in \C{}:  \,\,  e^\xi \in \Omega\}.
$$
The set $U$ is the preimage of $\Omega \setminus \{a\}$ by the exponential map (recall we assume $a=0$).  Defining $v:U \to \mathcal P$ and $\psi : U \to M_a^3(\R{})$ by
$$
v(\xi) = u(e^\xi) \,\,\,\  \mbox{and}\,\,\,\,  \psi(\xi) = \phi(e^\xi),
$$
we see that both $v(\xi) = v(\xi + 2\pi i)$ and $\psi(\xi) = \psi(\xi+2\pi i)$ whenever $e^{\xi} \in \Omega$, that is, both $v$ and $\psi$ are $2\pi i$-periodic.  Denote now
$$
H = \{\xi \in U : -\pi < \mathrm{Im}(\xi) < \pi\}, 
$$
and
$$
H_{\lambda} = \{\xi \in H : \,\,\,  \mathrm{Re}(\xi)\leq \lambda\},
$$
for $\lambda \in \R{}$.  We observe that, since $\phi \in (W^{1, 2}\cap L^\infty)(\Omega, M_a^3(\R{}))$, then $\psi \in L^\infty(H, M_a^3(\R{}))$ and $\nabla \psi \in L^2(H, M_a^3(\R{}))$.
We will now prove the following
\begin{proposition}\label{properties_psi}
The map $\psi: U \to M_a^3(\R{})$ satisfies $\nabla \psi \in (W^{1, 2}\cap L^\infty)(H, M_a^3(\R{}))$, and
$$
\nabla \psi(\xi_1, \xi_2) \to 0 \,\,\,\,\mbox{as}\,\,\,\,  \xi_1 \to -\infty, 
$$
uniformly in $\xi_2$. 
\end{proposition}
\begin{remark}\label{gradient_phi_decent}
Notice that this proposition, along with the relation $\psi(\xi) = \phi(e^{\xi})$, allow us to conclude that 
$$
\lim_{r\to 0}  \left ( r \sup_{x\in \partial B_r(a)} \abs{\nabla \phi(x)} \right ) = 0.
$$
\end{remark}
\begin{proof}[Proof of Proposition \ref{properties_psi}] The proof follows an argument that appears in \cite{BBH0}.

\medskip
\medskip

\noindent {\bf Step 1.}  The map $\psi$ satisfies the inequality
$$
-\Delta \left (\abs{\nabla \psi}^2  \right ) + \abs{D^2\psi}^2 \leq C ( \abs{\nabla \psi}^2 + \abs{\nabla \psi}^4)
$$
for some constant $C>0$.

\begin{proof}[Proof of Step 1.]
Direct computations show that
\begin{equation}\label{current_v}
j(v) = [v, \nabla v] = \frac{e_2}{2}\Lambda + \nabla^\perp \psi,
\end{equation}
where $e_2 = (0,1)$ is the second vector of the canonical basis in $\R{2}$.  Now $v$ is a projection-valued map.  This has the consequence that
$$
\left [  \frac{\partial v}{\partial \xi_1}, \frac{\partial v}{\partial \xi_2} \right ] = -\left [  \left [  v, \frac{\partial v}{\partial \xi_1} \right ] , \left [  v, \frac{\partial v}{\partial \xi_2} \right ]   \right ],
$$
and from here we deduce
$$
- \Delta \psi = \nabla^\perp \cdot j(v) = 2\left [  \frac{\partial v}{\partial \xi_1}, \frac{\partial v}{\partial \xi_2} \right ] = -2\left [  \left [  v, \frac{\partial v}{\partial \xi_1} \right ] , \left [  v, \frac{\partial v}{\partial \xi_2} \right ]   \right ].
$$
From the equation \eqref{current_v} we obtain
$$
\Delta \psi = 2\left [  \frac{\partial \psi}{\partial \xi_2}, \frac{\Lambda}{2} - \frac{\partial \psi}{\partial \xi_1}\right ] =2 \left [  \frac{\partial \psi}{\partial \xi_1}, \frac{\partial \psi}{\partial \xi_2}\right ] + \left [ \frac{\partial \psi}{\partial \xi_2}, \Lambda  \right ].
$$
From here we derive the identity
\begin{align}
-\Delta \left (\frac{\abs{\nabla \psi}^2}{2}  \right ) + \abs{D^2\psi}^2 = 2 \langle \left [ \frac{\partial \psi}{\partial \xi_1}, \frac{\partial \psi}{\partial \xi_2}  \right ], \Delta \psi \rangle + \langle \Lambda, \left [ \frac{\partial \psi}{\partial \xi_1}, \frac{\partial^2 \psi}{\partial \xi_1\partial \xi_2}\right ] + \left [  \frac{\partial \psi}{\partial \xi_2}, \frac{\partial^2 \psi}{\partial \xi_2^2} \right ]  \rangle. \label{laplacian_nabla_psi_sqr}
\end{align}
We now observe that
$$
\langle \left [ \frac{\partial \psi}{\partial \xi_1}, \frac{\partial \psi}{\partial \xi_2}  \right ], \Delta \psi \rangle \leq \abs{\left [ \frac{\partial \psi}{\partial \xi_1}, \frac{\partial \psi}{\partial \xi_2}  \right ]}\abs{\Delta \psi} \leq C(\delta) \abs{\nabla \psi}^4 + \delta\abs{D^2 \psi}^2.
$$
Similarly, 
$$
\langle \Lambda, \left [ \frac{\partial \psi}{\partial \xi_1}, \frac{\partial^2 \psi}{\partial \xi_1\partial \xi_2}\right ] + \left [  \frac{\partial \psi}{\partial \xi_2}, \frac{\partial^2 \psi}{\partial \xi_2^2} \right ]  \rangle \leq C\abs{\nabla \psi}\abs{D^2\psi} \leq C(\delta) \abs{\nabla \psi}^2 + \delta \abs{D^2 \psi}^2.
$$
For $\delta > 0$ small enough, we can absorb the terms $\delta \abs{D^2 \psi}^2$ in the term $\abs{D^2 \psi}^2$ on the left hand side of \ref{laplacian_nabla_psi_sqr}.  The conclusion of Step 1 follows from here.
\end{proof}

\noindent {\bf Step 2.}  $\nabla \psi \in W^{1, 2}(H, M_a^3(\R{}))$.

\begin{proof}[Proof of Step 2.]  Let $M > L + 1> 0$, and consider a cut-off function $\chi_{L, M}  \in C^\infty_0(\C{})$ such that $0 \leq \chi_{L, M} \leq 1$ and
$$
\chi_{L, M}(\xi) = \left \{   \begin{array}{cc} 1 & \mbox{for}\,\,\,   \xi_1 \in [-M, -L] \,\,\,\mbox{and}\,\,\,  \xi_2 \in [-\pi, \pi]  \\
0 & \xi_1 \notin [-M-1, -L+1] \,\,\,\mbox{or}\,\,\,  \xi_2 \notin [-2\pi, 2\pi]. \end{array} \right .
$$
We choose $M, L$ large enough so that the support ${\rm supp}(\chi_{L, M})\subset U$.  Note that the definition of $\chi_{L, M}$ allows us to require that the derivatives of $\chi_{L, M}$, up to order $2$, be bounded uniformly in $L, M>0$.  From Step 1 we obtain
\begin{equation}\label{step_1_sobolev_est_psi}
\int_U \chi_{L, M}^2 \abs{D^2 \psi}^2 \leq C\int_{{\rm supp}(\chi_{L, M})} \abs{\nabla \psi}^2 + C\int_{U} \chi_{L, M}^2 \abs{\nabla \psi}^4.
\end{equation}
We will estimate the last integral in this inequality using Gagliardo-Nirenberg.  To this end, we first notice that
$$
\int_{B_r(a)}\abs{\nabla \phi}^2 = \int_{H_{\ln(r)}} \abs{\nabla \psi}^2,
$$
where we recall that
$$
H_{\ln(r)} = \{\xi\in \C{}: e^\xi \in \Omega , \,\, \xi_1\leq \ln(r), \xi_2\in [-\pi, \pi]\}.
$$
Since $\phi \in W^{1, 2}(\Omega)$, it holds 
$$
\lim_{r\to 0} \int_{B_r(a)}\abs{\nabla \phi}^2 = 0,
$$
and hence
$$
\lim_{L \to -\infty}  \left ( \lim_{M \to -\infty} \int_{{\rm supp}(\chi_{L, M})} \abs{\nabla \psi}^2 \right )  = 0.
$$
Next, we recall that Gagliardo-Nirenberg inequality establishes that
$$
\int_{\C{}} f^2 \leq \left (  \int_{\C{}} \abs{\nabla f} \right )^2
$$
for any function $f \in C^1_0(\C{})$.  We apply this estimate to $f=\chi_{L, M}\abs{\nabla \psi}^2$.  Clearly
$$
\abs{\nabla f} \leq C ( \1_{{\rm supp}(\chi_{L, M})}\abs{\nabla \psi}^2 + \chi_{L, M} \abs{\nabla \psi}\abs{D^2 \psi}),
$$
where $\1_{A}$ denotes the characteristic function of the set $A$.  From here we obtain
$$
\int_{U} \chi_{L, M}^2 \abs{\nabla \psi}^4 \leq C \left ( \int_{{\rm supp}(\chi_{L, M})} \abs{\nabla \psi}^2 \right )^2 + C\int_{{\rm supp}(\chi_{L, M})} \abs{\nabla \psi}^2 \int_{U} \chi_{L, M}^2\abs{D \psi}^2.
$$
Since $\mathop{\lim}\limits_{L \to -\infty}  \left ( \mathop{\lim}\limits_{M \to -\infty} \int_{{\rm supp}(\chi_{L, M})} \abs{\nabla \psi}^2 \right )  = 0$, choosing $L>0$ large enough the last estimate and \eqref{step_1_sobolev_est_psi} yield
\begin{equation}
    \label{eq:in}
    \int_{U}\chi_{L, M}^2 \abs{D^2 \psi}^2 \leq C \left ( \int_{{\rm supp}(\chi_{L, M})} \abs{\nabla \psi}^2 \right )^2 + C\int_{{\rm supp}(\chi_{L, M})} \abs{\nabla \psi}^2.
\end{equation}
Finally, we let $M \to -\infty$ in \eqref{eq:in}, and recall that $\psi$ is $2\pi i$-periodic.  We obtain
$$
\int_{H_{-L}} \abs{D^2 \psi}^2 \leq C\left ( \int_{H_{-L+1}} \abs{\nabla \psi}^2 \right )^2 + C\int_{H_{-L+1}} \abs{\nabla \psi}^2
$$
for some constant $C > 0$ that remains bounded as $L \to -\infty$.  Now, we know the map $\phi\in W^{1, 2}(\Omega, M_a^3(\R{}))$ is smooth away from $a\in \Omega$.  We conclude that $D^2\psi \in L^2(H, M_a^3(\R{}))$, which concludes the proof of Step 2.
\end{proof}

\noindent  {\bf Step 3.}  $\nabla \psi(\xi_1, \xi_2) \to 0$ as $\xi_1 \to -\infty$, uniformly in $\xi_2\in \R{}$.  

\begin{proof}[Proof of Step 3.] To show this we first recall that in our last step we proved the estimate
$$
\int_{H_{-L}} \abs{D^2 \psi}^2 \leq C\left ( \int_{H_{-L+1}} \abs{\nabla \psi}^2 \right )^2 + C\int_{H_{-L+1}} \abs{\nabla \psi}^2.
$$
We can use this inequality to conclude that
$$
\lim_{L\to -\infty} \int_{H_{-L}} \abs{D^2 \psi}^2 = 0.
$$
Now the map $\psi$ is $2\pi i$-periodic, which is to say that it is $2\pi$-periodic in the variable $\xi_2$.  Let $\xi_0 = (\xi_{0, 1}, \xi_{0, 2}) \in H_{-\lambda}$ and $R > 0$ be such that $B_{2R}(\xi_{0}) = B_{2R}((\xi_{0, 1}, \xi_{0, 2}))\subset H_{-\lambda}$. This implies that for a fixed $R>0$, we have
$$
\lim_{\xi_{0, 1}\to -\infty} \int_{B_{2R}((\xi_{0, 1}, \xi_{0, 2}))} (\abs{D^2 \psi}^2 + \abs{\nabla \psi}^2 ) = 0.
$$
By standard Sobolev embeddings, we conclude that
$$
\lim_{\xi_{0, 1}\to -\infty} \int_{B_{2R}((\xi_{0, 1}, \xi_{0, 2}))} \abs{\nabla \psi}^p = 0
$$
for any $1 < p < \infty$.  Next we recall from Step 1 that
$$
-\Delta \left (\abs{\nabla \psi}^2  \right ) + \abs{D^2\psi}^2 \leq C ( \abs{\nabla \psi}^2 + \abs{\nabla \psi}^4)
$$
for some constant $C>0$.  Fixing $p > 2$, Theorem 8.17 of \cite{GT} yields
$$
\mathop{\sup}\limits_{B_R((\xi_{0, 1}, \xi_{0, 2}))} \abs{\nabla \psi}^2 \leq C\int_{B_{2R}((\xi_{0, 1}, \xi_{0, 2}))} \abs{\nabla \psi}^{2p} + C\int_{B_{2R}((\xi_{0, 1}, \xi_{0, 2}))} \abs{\nabla \psi}^{4p},
$$
for some $C > 0$ that depends on $R>0$ and $p>2$.  Since
$$
\lim_{\xi_{0, 1}\to -\infty} \int_{B_{2R}((\xi_{0, 1}, \xi_{0, 2}))} \abs{\nabla \psi}^p = 0
$$
for any $1 < p < \infty$, this proves Step 3, and concludes the proof of the proposition.
\end{proof}\renewcommand{\qedsymbol}{}
\end{proof}

Now recall that
$$
\Lambda = \left (\begin{array}{ccc} 0 & -1 & 0 \\ 1 & 0 & 0 \\ 0 & 0 & 0  \end{array} \right ),
$$
and define
\begin{equation}\label{canonical_flat}
u_0(x) = \frac{1}{2} \left (  \left (  \begin{array}{ccc} 1 & 0 & 0 \\ 0 & 1 & 0 \\ 0 & 0 & 0  \end{array} \right ) + \left (  \begin{array}{ccc} \cos(\theta(x)) & \sin(\theta(x)) & 0 \\ \sin(\theta(x)) & -\cos(\theta(x)) & 0 \\ 0 & 0 & 0  \end{array} \right ) \right ),
\end{equation}
where $\theta(x)$ is the standard angular variable from polar coordinates centered at $a\in \R{2}$.  A direct computation shows that
$$
j(u_0) = [u_0, \nabla u_0] = \frac{\hat{\theta}}{2r}\Lambda,
$$
We will refer to $u_0$ as a {\it canonical flat map} $u_0:\Omega \setminus \{a\} \to \mathcal P$ represented by $\Lambda$.  This is the same as saying that the image of $u_0$ in $\mathcal{P}$ is a closed geodesic.  Observe that the image $u_0(\partial B_r(a))$ of every circle centered at $a$ by $u_0$ is a closed geodesic in $\mathcal P$.

\medskip
\medskip

Let us observe that the images of the matrices $u_0(x)$, $x\in \Omega \setminus \{a\}$, are all contained in a fixed plane, that we will denote by $S_0$.  Now we state the following proposition.
\begin{proposition}
With the notation above, we have
$$
\frac{1}{\abs{x-a}} [\Lambda , [u, u_0]] \in L^2(\Omega, M_a^3(\R{})) \,\,\,\,\,   \mbox{and}\,\,\,\,\,  [u, u_0] \in W^{1, 2}(\Omega, M_a^3(\R{})).
$$
\end{proposition}

\begin{proof}

We start by recalling that
$$
\nabla u = [u, j(u)], \,\,\, \nabla u_0 = [u_0, j(u_0)].
$$
Because of this we obtain
\begin{align*}
\nabla [u, u_0] &= [\nabla u, u_0] + [u, \nabla u_0] \\  
&= \left [ \left [u, \frac{\hat{\theta}}{2r} \Lambda + \nabla^\perp \phi \right ], u_0 \right ]  + \left [u,  \left [u_0, \frac{\hat{\theta}}{2r} \Lambda  \right ], \right ] \\
&= \frac{\hat{\theta}}{2r}  \left (\left [ [u, \Lambda], u_0 \right ] + [[\Lambda, u_0], u]\right ) + [[u, \nabla^\perp \phi], u_0]
\\&= \frac{\hat{\theta}}{2r}[[u, u_0], \Lambda] + [[u, \nabla^\perp \phi], u_0],
\end{align*}
where the last identity follows by the Jacobi identity for commutators.  Summarizing, we have
\begin{equation}\label{grad_comm_vect}
\nabla [u, u_0] = \frac{\hat{\theta}}{2r} [[u, u_0], \Lambda] +  [[u, \nabla^\perp \phi], u_0].
\end{equation}
Since $\phi \in W^{1, 2}(\Omega, M_a^3(\R{}))$, this last identity shows that
$$
\frac{1}{\abs{x-a}} [\Lambda, [u, u_0]]\in L^2(\Omega, M_a^3(\R{})) \,\,\,\,  \iff \,\,\,\, [u, u_0] \in W^{1, 2}(\Omega, M_a^3(\R{})).
$$
Hence, we concentrate on proving $\frac{1}{\abs{x-a}} [\Lambda, [u, u_0]]\in L^2(\Omega, M_a^3(\R{}))$.

\medskip
\medskip

Take the commutator of \eqref{grad_comm_vect} with $\Lambda$, and then take the inner product of the resulting equation with $[\Lambda, [u, u_0]]$.  We obtain
\begin{equation}\label{grad_abs_commutator}
\nabla \left (  \frac{\abs{[\Lambda, [u, u_0]]}^2}{2}  \right ) = \langle [\Lambda, [u, u_0]], [\Lambda,  [[u, \nabla^\perp \phi], u_0]] \rangle.
\end{equation}

\medskip
\medskip

Next consider $r>0$ such that $\overline{B_r(a)} \subset \Omega$, and observe that at some $x_r \in \partial B_r(a)$ the image of $u(x_r)$ will be contained in the plane $S_0$ in which the images of $u_0(x)$ are contained.  Because of this, we deduce that
$$
[\Lambda, [u(x_r), u_0(x_r)]] = 0.
$$
For any $x\in \partial B_r(a)$, let $\gamma_{x_r, x}$ be an arc of $\partial B_r(a)$ from $x_r$ to $x$.  Integrating \eqref{grad_abs_commutator} over $\gamma_{x_r, x}$ we obtain
$$
\frac{\abs{[\Lambda, [u, u_0]]}^2}{2}(x) = \int_{\gamma_{x_r, x}} \langle [\Lambda, [u, u_0]], [\Lambda,  [[u, \nabla^\perp \phi], u_0] \rangle \cdot \tau \,dl,
$$
which gives
$$
\abs{[\Lambda, [u, u_0]]}^2(x) \leq C \int_{\partial B_r(a)}  \abs{[\Lambda, [u, u_0]]} \abs{[u, \nabla^\perp \phi]}\,dl
$$
for every $x\in \partial B_r(a)$.  Integrating the last inequality over $\partial B_r(a)$, and dividing by $r^2$, we find
$$
\frac{1}{r^2} \int_{\partial B_r(a)} \abs{[\Lambda, [u, u_0]]}^2 \leq \frac{C}{r} \int_{\partial B_r(a)}  \abs{[\Lambda, [u, u_0]]} \abs{[u, \nabla^\perp \phi]}.
$$
From here we obtain
$$
\frac{1}{r^2} \int_{\partial B_r(a)} \abs{[\Lambda, [u, u_0]]}^2 \leq C \int_{\partial B_r(a)}  \abs{[u, \nabla^\perp \phi]}^2
$$
for every $r>0$ such that $\overline{B_r(a)} \subset \Omega$ and some $C>0$ independent of $r>0$.  Since $\phi \in W^{1, 2}(\Omega, M_a^3(\R{}))$, and $\abs{u}\leq 1$, this last inequality implies 
$$
\frac{1}{\abs{x-a}} [\Lambda, [u, u_0]]\in L^2(\Omega, M_a^3(\R{})).
$$
This concludes the proof of the proposition.

\end{proof}

Now recall that we denote by $S_0$ the plane that contains all the images of $u_0(x)$, $x\in \Omega \setminus \{a\}$, and denote by $P_3$ the orthogonal projection onto the $1$-dimensional subspace in $\R{3}$ orthogonal to $S_0$.  The matrix $P_3$ can be characterized as the orthogonal projection onto the kernel of $\Lambda$, or alternatively, as the only element $P_3\in \mathcal P$ such that $P_3 \Lambda = \Lambda P_3 = 0$.  In particular, $[P_3, \Lambda] = 0$.  A similar argument to the one we used in the last proposition yields the following

\begin{proposition}\label{commutator_u_P_3_in_L_2}
With the notation above, we have
$$
\frac{1}{r}[u, P_3]\in L^2(\Omega, M^3_a(\R{})), \,\,\,\,  [u, P_3]\in W^{1, 2}(\Omega, M^3_a(\R{})),
$$
and
$$
\lim_{r\to 0} \sup_{x\in \partial B_r(a)} \abs{[u(x), P_3]} = 0.
$$
\end{proposition}

\begin{proof}
We begin by recalling that
$$
\nabla u = \frac{\hat{\theta}}{2r}[u, \Lambda] + [u, \nabla^\perp \phi].
$$
Taking commutator of this last identity with $P_3$ we obtain
\begin{equation}\label{commutator_u_P_3}
\nabla [u, P_3] = \frac{\hat{\theta}}{2r}[[u, P_3], \Lambda] + [[u, \nabla^\perp \phi], P_3],
\end{equation}
where in the first term of the right hand side we used Jacobi's identity for commutators, plus the fact that $[\Lambda, P_3]=0$.  As in the proof of the previous proposition, this last identity shows that
$$
\frac{1}{\abs{x-a}} [[u, P_3], \Lambda]\in L^2(\Omega, M_a^3(\R{})) \,\,\,\,  \iff \,\,\,\, [u, P_3] \in W^{1, 2}(\Omega, M_a^3(\R{})).
$$
We will, in fact, prove the stronger statement
\begin{equation}\label{commutator_compares_radius}
\frac{1}{\abs{x-a}} [u, P_3] \in L^2(\Omega, M_a^3(\R{})).
\end{equation}
Clearly this will establish the first two statements of the proposition.  To demonstrate \eqref{commutator_compares_radius}, take the inner product of \eqref{commutator_u_P_3} with $[u, P_3]$, to obtain
\begin{equation}\label{grad_com_u_P_3}
\nabla \left ( \frac{\abs{[u, P_3]}^2}{2} \right ) = \langle [[u, \nabla^\perp \phi], P_3], [u, P_3]\rangle.
\end{equation}
Now let $r>0$ be such that $\overline{B_r(a)} \subset \Omega$, and recall that at some $x_r \in \partial B_r(a)$ the image of $u(x_r)$ will belong to the plane $S_0$ in which the images of $u_0(x)$ are contained.  Because of this, we deduce that
$$
[u(x_r), P_3] = 0.
$$
For any $x\in \partial B_r(a)$, let $\gamma_{x_r, x}$ be an arc of $\partial B_r(a)$ from $x_r$ to $x$.  Integrating equation \eqref{grad_com_u_P_3} over $\gamma_{x_r, x}$ we get
$$
\frac{\abs{[u, P_3]}^2}{2}(x) = \int_{\gamma_{x_r, x}} \langle [u, P_3], [[u, \nabla^\perp \phi], P_3] \rangle \cdot \tau \,dl \leq C\int_{\partial B_r(a)}\abs{[u, P_3]}\abs{[u, \nabla^\perp \phi]},
$$
for every $x\in \partial B_r(a)$.  A similar argument to the one we used in the previous proposition yields
$$
\frac{1}{\abs{x-a}} [u, P_3] \in L^2(\Omega, M_a^3(\R{})).
$$
Also from
$$
\frac{\abs{[u, P_3]}^2}{2}(x)  \leq C\int_{\partial B_r(a)}\abs{[u, P_3]}\abs{[u, \nabla^\perp \phi]}
$$
we obtain
$$
\abs{[u, P_3]}^2(x)  \leq C\int_{\partial B_r(a)}\abs{\nabla^\perp \phi} \leq \frac{C}{r}\int_{\partial B_r(a)}\abs{y-a}\abs{\nabla^\perp \phi}\,dl(y)
$$
for all $x\in \partial B_r(a)$.  By Remark \ref{gradient_phi_decent}, this concludes the proof of the proposition.
\end{proof}

Our next result estimates the inner product $\langle u, P_3\rangle$.

\begin{proposition}\label{latitude_control}
With the notation above we have
$$
\frac{\langle u, P_3 \rangle(x)}{\abs{x-a}^2} \in L^1(\Omega).
$$
Furthermore
$$
\lim_{r\to 0} \sup_{x\in \partial B_r(a)} \abs{\langle u, P_3 \rangle(x)} = 0.
$$
\end{proposition}

\begin{proof}
From our previous computations we know that
$$
\nabla u = \frac{\hat{\theta}}{2} [u, \Lambda] + [u, \nabla^\perp \phi].
$$
Now recall the triple product formula for square matrices
$$
\langle [A, B], C\rangle = \langle [C^T, A], B^T\rangle = \langle [ B, C^T], A^T\rangle.
$$
Taking the inner product of $\nabla u = \frac{\hat{\theta}}{2} [u, \Lambda] + [u, \nabla^\perp \phi]$ with $P_3$, and using the triple product formula for the first term of the right hand side, plus the fact that $[\Lambda, P_3]=0$, we obtain
$$
\nabla \langle u, P_3 \rangle = \langle [u, P_3], \nabla^\perp \phi \rangle.
$$
Recall now that for $\overline{B_r(a)}\subset \Omega$ there is $x_r \in \partial B_r(a)$ such that $u(x_r)\in S_0$.  If $\gamma_{x_r, x}$ is an arc of $\partial B_x(a)$ from $x_r$ to $x\in \partial B_r(a)$, we obtain
$$
\langle u, P_3\rangle (x) = \int_{\gamma_{x_r, x}} \langle [u, P_3], \nabla^\perp\phi \cdot \tau \rangle.
$$
We deduce that
$$
\abs{\langle u, P_3\rangle (x)} \leq \int_{\partial B_r(a)}   \abs{ [u, P_3]}\abs{ \nabla^\perp\phi }
$$
for all $x\in \partial B_r(a)$.  By Remark \ref{gradient_phi_decent}, this implies $\mathop{\lim}\limits_{r\to 0} \mathop{\sup}\limits_{x\in \partial B_r(a)} \abs{\langle u, P_3 \rangle(x)} = 0$.  Now, integrating the last inequality over $\partial B_r(a)$ and dividing by $r^2$ we obtain
$$
\frac{1}{r^2} \int_{\partial B_r(a)} \abs{\langle u, P_3\rangle} \leq 2\pi \int_{\partial B_r(a)}  \frac{\abs{ [u, P_3]}}{r} \abs{ \nabla^\perp\phi}.
$$
Integrating over $r\in [0, R]$ for $R>0$ such that $B_R(a)\subset \Omega$, we find that
$$
\int_{B_R(a)} \frac{\abs{\langle u, P_3\rangle}}{\abs{x-a}^2}\leq \int_{B_R(a)}\frac{\abs{[u, P_3]}}{\abs{x-a}}\abs{\nabla \phi}.
$$
Since both $\nabla\phi, \,\, \frac{[u, P_3]}{\abs{x-a}}\in L^2(\Omega , M_a^3(\R{}))$, this concludes the proof.

\end{proof}

Now we show that the length of the curves $\gamma_r(e^{i\theta}) = u(re^{i\theta})$ approaches the length of closed geodesics in $\mathcal P$ as $r\to 0$.  This is the content of the next proposition.

\begin{proposition}
Let $L$ denote the length of a closed geodesic in $\mathcal P$.  We have
$$
\lim_{r\to 0} \int_0^{2\pi} \abs{\gamma_r} = L.
$$
\end{proposition}

\begin{proof}
Let us first recall that
$$
L = \int_0^{2\pi} \abs{\frac{du_0}{d\theta}}(\theta)\, d\theta = \sqrt{2}\pi.
$$
From \cite{GM} we have that
$$
\int_{\partial B_r(a)} \abs{\nabla u \cdot \tau}^2 = \int_{\partial B_r(a)}\abs{\nabla u_0\cdot \tau}^2 + \int_{\partial B_r(a)}\abs{\nabla u\cdot \nu}^2,
$$
where $\tau$ and $\nu$ denote the tangent and outer normal to $\partial B_r(a)$, respectively, and $u_0$ denotes a canonical flat map.  Now, we also know that
$$
j(u) = \frac{\hat{\theta}}{2r}\Lambda + \nabla^\perp \phi,
$$
from which we deduce that
$$
\nabla u = \frac{\hat{\theta}}{2r}[u, \Lambda] + [u, \nabla^\perp \phi].
$$
We obtain on $\partial B_r(a)$ that $\nabla u \cdot \nu = [u, \nabla^\perp \phi \cdot \nu]$.  With all this we now have
\begin{align*}
\frac{L^2}{2\pi r} &\leq \frac{1}{2\pi r}\left ( \int_{\partial B_r(a)}\abs{\nabla u \cdot \tau} \right )^2 \\ &\leq \int_{\partial B_r(a)}\abs{\nabla u \cdot \tau}^2 = \int_{\partial B_r(a)}\abs{\nabla u_0\cdot \tau}^2 + \int_{\partial B_r(a)}\abs{\nabla u\cdot \nu}^2 \\
&= \frac{\pi}{r} +  \int_{\partial B_r(a)}\abs{[u, \nabla^\perp \phi \cdot \nu]}^2 = \frac{L^2}{2\pi r}+  \int_{\partial B_r(a)}\abs{[u, \nabla^\perp \phi \cdot \nu]}^2.
\end{align*}
Therefore
$$
L^2 \leq \left ( \int_0^{2\pi} \abs{\frac{d  \gamma_r}{d\theta}} \, d\theta \right )^2 \leq L^2 + \frac{2\pi}{r}   \int_{\partial B_r(a)}\abs{y-a}^2\abs{[u, \nabla^\perp \phi \cdot \nu]}^2 \, dl(y).
$$
We now appeal to Remark \ref{gradient_phi_decent} to finish the proof.
\end{proof}

\medskip
\medskip
\medskip
\medskip

Next we use the last proposition to write the limit of minimizers $u$ in terms of two angles near the singularity.  Neither of these angles will be single-valued in $\Omega\setminus \{a\}$, hence we do this for the map $v(\xi) = u(e^\xi)$ in the domain
$$
G_{-\lambda} = \{\xi\in \C{}: {\rm Re}(\xi) < -\lambda\}.
$$
Here we choose $\lambda > 0$ at least large enough for $G_{-\lambda}\subset U = \{\xi \in \C{}:  e^\xi \in \Omega\}$.  An additional condition on $\lambda$ will appear in the next proposition.

\begin{proposition}\label{latitude_longitude}
There is $\lambda >0$ large enough, and real valued functions $\alpha, \beta : G_{-\lambda}\to \R{}$ such that
$$
v(\xi) = n(\xi)n^T(\xi), 
$$
where
$$
n(\xi) = \left (  \begin{array}{c}  \cos(\beta)\cos(\alpha) \\ \cos(\beta)\sin(\alpha)\\ \sin(\beta) \end{array} \right ).
$$
Furthermore, $\beta(\xi + 2\pi i) = -\beta(\xi)$ and $\alpha(\xi+2\pi i) = \alpha(\xi) + \pi$ for all $\xi\in G_{-\lambda}$.  Finally, the function $\alpha$ can be written as $\alpha(\xi) = \frac{\xi_2 + \alpha_1(\xi)}{2}$ for some real valued function $\alpha_1:G_{-\lambda}\to \R{}$ such that $\alpha_1(\xi + 2\pi i) = \alpha_1(\xi)$.

\end{proposition}
\begin{remark}
Notice that the last proposition shows that the vector field
$$
k(\xi) = \left (  \begin{array}{c}  -\sin(\beta)\cos(\alpha) \\ -\sin(\beta)\sin(\alpha)\\ \cos(\beta) \end{array} \right )
$$
satisfies $v(\xi) k(\xi) = 0$ for $\xi \in G_{-\lambda}$.  Also, since $\beta(\xi + 2\pi i) = -\beta(\xi)$, and $\alpha(\xi+2\pi i) = \alpha(\xi) + \pi$, it follows that $k(\xi + 2\pi i) = k(\xi)$.  By our change of variables, this gives a map, that we still call $k \in W^{1, 2}(B_r(a), \mathbb{S}^2)$, such that $u(x)k(x)=0$ for all $x\in B_r(a)\setminus \{a\}$.  Furthermore, we can extend $k$ to all of $\Omega$ while still adhering to $u(x) k(x) = 0$ for $x \in \Omega\setminus \{a\}$.  By our change of variables, and the properties of $\alpha$ and $\beta$ stated in Step 1 of Theorem \ref{near_sing}, this defines a single-valued map, that we still denote by $k \in W^{1, 2}(\Omega, \mathbb{S}^2)$, such that
$$
u(x) k(x) = 0 \,\,\,\mbox{for all}\,\,\,  x\in \Omega \setminus \{a\}.
$$
This proves the first claim of Proposition \ref{convergence_geodesic_third_eval}.
\end{remark}

\begin{proof}[Proof of Proposition \ref{latitude_longitude}.]
First observe that Proposition \ref{latitude_control} ensures that for $s=\frac{1}{2}$, there is $r>0$ small such that
$$
\abs{\langle u, P_3\rangle}(x) < s,
$$
for all $x\in B_r(a)\setminus \{a\}$.  Let us now recall the definition $v(\xi) = u(e^\xi)$, and note that this last condition implies the existence of $\lambda > 0$ large enough such that
$$
\abs{\langle v, P_3\rangle}(\xi) < s
$$
for every $\xi \in G_{-\lambda}$ where
$$
G_{-\lambda} = \{\xi\in \C{}: {\rm Re}(\xi) < -\lambda\}.
$$
Now $G_{-\lambda}$ is simply-connected, so we can lift $v$ through $n:G_{-\lambda}\to \mathbb{S}^2$, that is, we can find a map $n:G_{-\lambda}\to \mathbb{S}^2$ such that
$$
v(\xi) = n(\xi)n(\xi)^T
$$
for every $\xi\in G_{-\lambda}$.  Now let
$$
n(\xi) = \left ( \begin{array}{c} n_1(\xi) \\ n_2(\xi) \\ n_3(\xi) \end{array}\right ),
$$
and observe that $n_3^2(\xi) = \langle v(\xi), P_3\rangle < s = \frac{1}{2}$ for every $\xi \in G_{-\lambda}$.  In particular, we can define $\beta : G_{-\lambda}\to \R{}$ by the condition $\sin(\beta(\xi)) = n_3(\xi)$.  Since $n_3^2(\xi) = \langle v(\xi), P_3\rangle < s = \frac{1}{2}$, we conclude that $\cos(\beta(\xi)) > \frac{1}{\sqrt{2}}$ for every $\xi \in G_{-\lambda}$.  In particular, the map $\zeta : G_{-\lambda}\to \mathbb{S}^1$ defined by
$$
\zeta(\xi) = \left ( \begin{array}{c}   \sec(\beta) \, n_1 \\ \sec(\beta)\,  n_2 \end{array} \right )
$$
is well defined for $\xi \in G_{-\lambda}$.  Again, since $G_{-\lambda}$ is simply-connected, we cal lift $\zeta$ through a real valued function $\alpha$ in such a way that
$$
\zeta(\xi) = \left ( \begin{array}{c}   \cos(\alpha) \\ \sin(\alpha) \end{array} \right ).
$$
Summarizing so far, we have found real valued functions $\alpha, \beta : G_{-\lambda}\to \R{}$ such that
$$
v(\xi) = n(\xi)n^T(\xi),
$$
where
$$
n(\xi) = \left ( \begin{array}{c}  \cos(\beta(\xi)) \cos(\alpha(\xi)) \\ \cos(\beta(\xi))\sin(\alpha(\xi)) \\  \sin(\beta(\xi)) \end{array}\right ).
$$
To prove the last conclusions of the proposition let us recall that $u:\Omega \setminus \{a\}\to \mathcal P$ restricted to any circle $\partial B_r(a)\subset \Omega$ represents a non-contractible curve in $\mathcal P$.  Since $v(\xi) = u(e^\xi) = n(\xi)n^T(\xi)$, we conclude that  $n(\xi + 2\pi i) = -n(\xi)$.  This shows that $\beta(\xi+2\pi i) = -\beta(\xi)$.

\medskip
\medskip

For the last conclusion let us recall the canonical flat map
$$
u_0(x) = \frac{1}{2} \left (  \left (  \begin{array}{ccc} 1 & 0 & 0 \\ 0 & 1 & 0 \\ 0 & 0 & 0  \end{array} \right ) + \left (  \begin{array}{ccc} \cos(\theta(x)) & \sin(\theta(x)) & 0 \\ \sin(\theta(x)) & -\cos(\theta(x)) & 0 \\ 0 & 0 & 0  \end{array} \right ) \right ).
$$
By hypothesis, we know that $u_0$ is homotopic to $u$ in $\Omega\setminus \{a\}$.  This implies that $v(\xi) = u(e^\xi)$ is homotopic to
$$
v_0(\xi) = \frac{1}{2} \left (  \left (  \begin{array}{ccc} 1 & 0 & 0 \\ 0 & 1 & 0 \\ 0 & 0 & 0  \end{array} \right ) + \left (  \begin{array}{ccc} \cos(\xi_2) & \sin(\xi_2) & 0 \\ \sin(\xi_2) & -\cos(\xi_2) & 0 \\ 0 & 0 & 0  \end{array} \right ) \right )
$$
in $G_{-\lambda}$.  Furthermore, a direct computation shows that $v_0$ lifts through $\mathbb{S}^2$ by the map
$$
n_0(\xi) = \left ( \begin{array}{c}  \cos(\xi_2/2)  \\ \sin(\xi_2/2) \\  0 \end{array}\right ).
$$
Now in $G_{-\lambda}$ we have $\sin^2(\beta) < s = \frac{1}{2}$.  Hence
$$
\zeta(\xi) = \left ( \begin{array}{c}   \cos(\alpha) \\ \sin(\alpha) \end{array} \right )
$$
is homotopic in $G_{-\lambda}$ to
$$
\zeta_0(\xi) = \left ( \begin{array}{c}   \cos(\xi_2/2) \\ \sin(\xi_2/2) \end{array} \right )
$$
We conclude that the function $\frac{\alpha_1(\xi)}{2} = \alpha(\xi) - \frac{\xi_2}{2}$ satisfies $\alpha_1(\xi+2\pi i) = \alpha_1(\xi)$.  This concludes the proof of this proposition.

\end{proof}

\medskip
\medskip
\medskip
\medskip

In our next step we recall a property of the Hopf differential.  Let us first recall its definition.  For this we will need to switch to complex derivatives in the plane.  We will use the usual
$$
\frac{\partial}{\partial z} = \frac{1}{2}\left ( \frac{\partial}{\partial x_1}- i \frac{\partial}{\partial x_2} \right ) \,\,\, \mbox{and} \,\,\,  \frac{\partial}{\partial \overline{z}} = \frac{1}{2}\left ( \frac{\partial}{\partial x_1} + i \frac{\partial}{\partial x_2} \right ).
$$
We will also denote
$$
j_{\C{}}(u) = \left [ u, \frac{\partial u}{\partial z} \right ]
$$
for the complex-valued current vector.  We have the relation
$$
j(u) = 2{\rm Re}\left ( j_{\C{}}(u) \right )e_1 - 2{\rm Im}\left ( j_{\C{}}(u) \right )e_2,
$$
where ${\rm Re}(z)$, ${\rm Im}(z)$ denote the real and imaginary parts of $z$, respectively.
\begin{definition}\label{def_hopf_diff}
For our limit of minimizers $u$, its Hopf differential is the function
$$
\omega_u(z) = {\rm tr}\left (\left (j_{\C{}}(u)\right )^2 \right ).
$$
By the properties of projection-valued maps, the Hopf differential can also be defined by
$$
\omega_u(z) = -{\rm tr} \left (\left (\frac{\partial u}{\partial z}\right )^2 \right ).
$$
\end{definition}
\begin{proposition}\label{Hopf_Diff_split}
For the Hopf differential of $u$ we have
$$
\omega_u(z) = -\frac{1}{8z^2} + h,
$$
where $h$ is a holomorphic map in all of $\Omega$.

\end{proposition}
\begin{proof}
To prove this proposition let us recall that our map $u$ is in fact
$$
u = \lim_{\eps \to 0} u_\eps,
$$
where the maps $u_\eps$ are global minimizers of the LdG energy in $\Omega$, and the convergence is strong in $W^{1, 2}_{loc}(\Omega\setminus \{a\}, M_s^3(\R{}))$.  Because of their minimizing character, the maps $u_\eps$ satisfy
$$
-\frac{\partial^2 u_\eps}{\partial z\partial \overline{z}} + \frac{1}{4\eps^2}\left (\nabla_u W\right )(u_\eps) = \lambda_\eps I_3,
$$
where $W(u)$ is the potential term in the energy, and $\lambda_\eps$ is the Lagrange multiplier associated to the restriction ${\rm tr}(u) = 1$.  Multiplying this equation by $\frac{\partial u_\eps}{\partial z}$, and taking trace of the resulting equation, we obtain
\begin{equation}\label{eq_Hopf_eps}
-\frac{\partial \omega_\eps}{\partial \overline{z}}  + \frac{\partial}{\partial z} \frac{W(u_\eps)}{2\eps^2} = 0,
\end{equation}
where we are using the notation
$$
\omega_\eps(z) = -{\rm tr} \left (\left (\frac{\partial u_\eps}{\partial z}\right )^2 \right ).
$$
Let now $G_\eps(z)$ be the convolution of $\frac{2W(u_\eps)}{\eps^2}$ with the Newtonian potential of the plane.  Note that this is well defined, since $W(u_\eps) = 0$ on $\partial \Omega$ by our boundary conditions there.  Note also that
$$
-\frac{\partial^2 G_\eps(z)}{\partial z\partial \overline{z}} = \frac{W(u_\eps)}{2\eps^2}.
$$
This and \eqref{eq_Hopf_eps} yield
$$
\frac{\partial }{\partial \overline{z}} \left ( \omega_\eps + \frac{\partial^2 G_\eps}{\partial z^2} \right ) = 0.
$$
We deduce that the function
$$
h_\eps =  \omega_\eps + \frac{\partial^2 G_\eps}{\partial z^2}
$$
is holomorphic in $\Omega$.  To conclude the proof we show that $h_\eps$ is uniformly bounded on compact sets $K\subset \Omega \setminus \{a\}$.  Since $h_\eps$ is holomorphic, this shows that $h_\eps$ is in fact locally bounded in $\Omega$, which then allows us to conclude that $h_\eps \to h$ along a subsequence for some $h$ holomorphic in $\Omega$.  To show that $h_\eps$ is uniformly bounded in compact sets $K\subset \Omega \setminus \{a\}$ we observe that the properties of $u_\eps$ we list in the appendix allow us to apply Step 1 of the proof of Theorem VII.1 of \cite{BBH}.  This concludes the proof of the proposition.
\end{proof}

\section{Estimates near the singularity}

We now will choose $r > 0$ small but independent of $\eps > 0$, and compare the energy of a minimizer of LdG in $B_r(a)$ with canonical flat data, to the energy of one of our minimizers in the same ball.  We recall that by canonical flat we mean data of the form given in equation \ref{canonical_flat}.  We will use the following notation:
\begin{equation}\label{energy_canonical_flat}
I(r, \eps) = \inf \{\int_{B_r(a)} e_\eps(u) : \,\,\,  u \,\,\, \mbox{is canonical flat on}\,\,\, \partial B_r(a)\}.
\end{equation}
\begin{theorem}\label{near_sing}
For $u$ our limit of minimizers, along a subsequence $\eps_n \to 0$ we can choose $r>0$ small but independent of $\eps_n$ so that
$$
\abs{I(r, \eps_n) - \int_{B_r(a)}e_{\eps_n}(u_{\eps_n})} \leq o(1) + q(r),
$$
where $o(1) \to 0$ as $\eps_n \to 0$, and
$$
q(r) \leq C\int_{B_{2r}(a)}\left ( \abs{\nabla \alpha_1}^2 + \abs{\nabla \beta}^2 + \frac{\abs{\nabla \alpha_1}}{\abs{x-a}} \right ) + \frac{C}{r^2}\int_{B_{2r}(a)} (\alpha_1^2+\beta^2),
$$
for some constant $C>0$ independent of $\eps_n$ and $r>0$.
\end{theorem}

\begin{remark}
In the terminology of the previous section, Step 2 in the following proof, along with Proposition \ref{latitude_control} show that $\beta(\xi_1, \xi_2) \to 0$ as $\xi_1 \to -\infty$, uniformly in $\xi_2$.  This, and Step 4 of the following proof, give the second claim of Proposition \ref{convergence_geodesic_third_eval}.  This also shows that the quantity $q(r)$ that appears in the statement of this theorem has $q(r) \to 0$ as $r\to 0$.

\end{remark}

\begin{proof}

We will prove the statement of the theorem in several steps.  For most of the proof we will let $$H_{-\lambda}=\left\{\xi\in\mathbb{C}:\mathrm{Re}(\xi)<-\lambda,-\pi\leq\mathrm{Im}(\xi)<\pi\right\}$$ be the lift of $B_r(a)$ through the exponential map.  In particular, $\lambda = \ln(\frac{1}{r})$ will be chosen in the course of the proof.  We also write $u = \mathop{\lim}\limits_{\eps_n \to 0} u_{\eps_n}$.  We know this convergence is strong in $W^{1, 2}(\Omega \setminus B_r(a), M^3_{s, 1}(\R{}))$.

\medskip
\medskip

\noindent {\bf Step 1.}  In this first step we assume $\xi \in H_{-\lambda}$ for $\lambda > 0$ chosen as in Proposition \ref{latitude_longitude}.  For such a $\lambda > 0$, the functions $\alpha_1, \beta$ from Proposition \ref{latitude_longitude} satisfy $\nabla \beta, \nabla \alpha_1 \in L^2(H_{-\lambda})$.

\medskip
\medskip

\begin{proof}[Proof of Step 1.]

To prove this, let us recall from Proposition \ref{latitude_longitude} that $v(\xi) = u(e^\xi) = n(\xi)n^T(\xi)$, where
$$
n(\xi) = \left ( \begin{array}{c}  \cos(\beta(\xi)) \cos(\alpha(\xi)) \\ \cos(\beta(\xi))\sin(\alpha(\xi)) \\  \sin(\beta(\xi)) \end{array}\right ) = \cos(\beta) n_0(\alpha(\xi)) + \sin(\beta(\xi)) e_3.
$$
Here
$$
n_0(\alpha) = \left ( \begin{array}{c}  \cos(\alpha) \\ \sin(\alpha) \\  0 \end{array}\right ) \,\,\,\mbox{and}\,\,\,  e_3 = \left ( \begin{array}{c}  0 \\ 0 \\  1 \end{array}\right ).
$$
We will also write
$$
m_0(\alpha) = \left ( \begin{array}{c}  -\sin(\alpha) \\ \cos(\alpha) \\  0 \end{array}\right ).
$$
A direct computation shows that
\begin{align*}
j(v) &= \left (n\frac{\partial n^T}{\partial \alpha} - \frac{\partial n}{\partial \alpha}n^T \right ) \nabla \alpha + \left (n\frac{\partial n^T}{\partial \beta} - \frac{\partial n}{\partial \beta}n^T \right ) \nabla \beta \\
&= \left (\cos^2(\beta)\Lambda + \cos(\beta)\sin(\beta) (e_3m_0(\alpha)^T - m_0(\alpha)e_3^T) \right )\nabla \alpha + \left (n_0(\alpha) e_3^T - e_3n_0(\alpha)^T \right ) \nabla \beta.
\end{align*}
We know that $\alpha(\xi) = \frac{\xi_2 + \alpha_1(\xi)}{2}$, so
$$
\nabla \alpha = \frac{e_2+\nabla \alpha_1}{2}.
$$
However, we also know from \cite{GM} that
$$
j(v) = \frac{e_2}{2}\Lambda + \nabla^\perp \psi.
$$
We deduce that
\begin{align*}
& \left (\cos^2(\beta)\Lambda + \cos(\beta)\sin(\beta) (e_3m_0(\alpha)^T - m_0(\alpha)e_3^T) \right )\nabla \alpha_1 + \left (n_0(\alpha) e_3^T - e_3n_0(\alpha)^T \right ) \nabla \beta \\
&= \left (\sin^2(\beta)\Lambda -  (\cos(\beta)\sin(\beta) (e_3m_0(\alpha)^T - m_0(\alpha)e_3^T) \right ) \frac{e_2}{2} + \nabla^\perp \psi,
\end{align*}
so that
$$
\cos^2(\beta) \abs{\nabla \alpha_1}^2 + \abs{\nabla \beta}^2 \leq C( \sin^2(\beta) + \abs{\nabla \psi}^2 ).
$$
Now recall that in the proof of Proposition \ref{latitude_longitude} we chose $\lambda>0$ so that $\sin^2(\beta) < \frac{1}{2}$ in $H_{-\lambda}$, so for $\xi \in H_{-\lambda}$ we have $\cos^2(\beta) > \frac{1}{2}$.  Finally, recall from Proposition \ref{latitude_control} that $\frac{\langle u, P_3\rangle }{\abs{x-a}^2}\in L^1(\Omega)$.  By our change of variables, this implies that $\langle v, P_3\rangle \in L^1(H_{-\lambda})$.  Since $\langle v, P_3\rangle = \sin^2(\beta)$ in $H_{-\lambda}$, this concludes the proof of this step.
\end{proof}

\medskip
\medskip

\noindent {\bf Step 2.}  We have $\nabla \alpha_1, \nabla \beta \in W^{1, 2}(H_{-\lambda})$ and
$$
\nabla \alpha_1(\xi_1, \xi_2) \to 0 \,\,\, \mbox{and}\,\,\, \nabla \beta(\xi_1, \xi_2)  \to 0
$$
as $\xi_1 \to -\infty$, both uniformly in $\xi_2 \in [-\pi, \pi]$.

\begin{proof}[Proof of Step 2.]  A direct computation shows that in $H_{-\lambda}$ we have
$$
\abs{\nabla v}^2 = \frac{\cos^2(\beta)}{2}\abs{e_2 + \nabla \alpha_1}^2 + 2\abs{\nabla \beta}^2.
$$
Since $u$ locally minimizes the Dirichlet integral in $\Omega \setminus \{a\}$, and our change of variables is holomorphic, we conclude that $v(\xi) = u(e^\xi)$ also locally minimizes the Dirichlet integral in $H_{-\lambda}$.  Because of this, $\alpha_1$ and $\beta$ satisfy
$$
\Delta \alpha_1 = 2\tan(\beta) (e_2 + \nabla \alpha_1 ) \cdot \nabla \beta
$$
and
$$
\Delta \beta =  -\frac{\sin(2\beta)}{8}\abs{e_2+\nabla \alpha_1}^2,
$$
respectively.  From here we find
\begin{align*}
\Delta \left ( \frac{\abs{\nabla \alpha_1}^2}{2}\right ) &= \abs{D^2 \alpha_1}^2  + 2\sec^2(\beta) \nabla \alpha_1 \cdot \nabla \beta (e_2+\nabla \alpha_1)\cdot \nabla \beta \\
&+ 2\tan(\beta) \sum_{k=1}^2 \frac{\partial \alpha_1}{\partial \xi_k} \nabla \left ( \frac{\partial \alpha_1}{\partial \xi_k} \right ) \cdot \nabla \beta \\
&+ 2\tan(\beta) \sum_{k=1}^2 \frac{\partial \alpha_1}{\partial \xi_k} (e_2 +  \nabla \alpha_1) \cdot \nabla \left ( \frac{\partial \beta}{\partial \xi_k} \right ) 
\end{align*}
and
\begin{align*}
\Delta \left ( \frac{\abs{\nabla \beta}^2}{2}\right ) &= \abs{D^2 \beta}^2  -\frac{\cos(2\beta)}{4}\left ( \abs{\nabla \alpha_1}^2 + 1 + 2\frac{\partial \alpha_1}{\partial \xi_2} \right ) \abs{\nabla \beta}^2 \\
&-\frac{\sin(2\beta)}{4} \sum_{k=1}^2 \frac{\partial \beta}{\partial \xi_k} \nabla \left ( \frac{\partial \alpha_1}{\partial \xi_k} \right ) \cdot (e_2 + \nabla \alpha_1).
\end{align*}
We conclude that
\begin{align*}
-\Delta \left ( \frac{\abs{\nabla \alpha_1}^2}{2} + \frac{\abs{\nabla \beta}^2}{2} \right ) &+ (1-\delta) ( \abs{D^2 \alpha_1}^2 + \abs{D^2 \beta}^2) \\
&\leq C(\delta) ( \abs{\nabla \alpha_1}^2+ \abs{\nabla \alpha_1}^4 + \abs{\nabla \beta}^2 + \abs{\nabla \beta}^4).
\end{align*}
We now follow Steps 2 and 3 of the proof of Proposition \ref{properties_psi} to complete the proof of this step.

\end{proof}

\noindent {\bf Step 3.}  $\nabla \alpha_1 \in L^1(H_{-\lambda})$.  By our change of variables, this implies that $\frac{\nabla \alpha_1}{\abs{x-a}} \in L^1(B_r(a))$.

\begin{proof}[Proof of Step 3.]  Recall from Proposition \ref{Hopf_Diff_split} that the Hopf differential of $u$ satisfies
$$
\omega_u(z) = -\frac{1}{8z^2}+ h(z),
$$
where $h$ is a holomorphic map in all of $\Omega$.  Now, by our change of variables from $\Omega$ to $U$, for the Hopf differential of $v$ we have
$$
\omega_v(\xi) = e^{2\xi} \omega_u(e^\xi) = -\frac{1}{8} + e^{2\xi}h(e^\xi).
$$
On the other hand, a direct computation in terms of $\alpha = \frac{\xi_2+\alpha_1}{2}$ and $\beta$ shows that
 \begin{align*}
j_{\C{}}(v) &= \left (n\frac{\partial n^T}{\partial \alpha} - \frac{\partial n}{\partial \alpha}n^T \right ) \frac{\partial \alpha}{\partial \xi} + \left (n\frac{\partial n^T}{\partial \beta} - \frac{\partial n}{\partial \beta}n^T \right ) \frac{\partial \beta}{\partial \xi} \\
&= \cos(\beta)\left (\cos(\beta)\Lambda + \sin(\beta) (e_3m_0(\alpha)^T - m_0(\alpha)e_3^T) \right )\frac{\partial \alpha}{\partial \xi} \\ &+ \left (n_0(\alpha) e_3^T - e_3n_0(\alpha)^T \right ) \frac{\partial \beta}{\partial \xi}.
\end{align*}
Since
$$
\frac{\partial \alpha}{\partial \xi} = -\frac{i}{4}+ \frac{\partial \alpha_1}{\partial \xi},
$$
we also have that
\begin{align*}
\omega_v(\xi) &= {\rm tr}\left ( \left ( j_{\C{}}(v)\right ) ^2 \right ) \\
&= 2\cos^2(\beta) \left ( -\frac{i}{4}+ \frac{\partial \alpha_1}{\partial \xi} \right )^2 + 2 \left ( \frac{\partial \beta}{\partial \xi} \right )^2 \\ &= -\frac{\cos^2(\beta)}{8} -i \cos^2(\beta)\frac{\partial \alpha_1}{\partial \xi} + 2 \cos^2(\beta) \left ( \frac{\partial \alpha_1}{\partial \xi} \right )^2 + 2 \left ( \frac{\partial \beta}{\partial \xi} \right )^2.
\end{align*}
We conclude that
$$
-i\cos^2(\beta) \frac{\partial \alpha_1}{\partial \xi} = -\frac{\sin^2(\beta)}{8} - 2 \cos^2(\beta) \left ( \frac{\partial \alpha_1}{\partial \xi} \right )^2 - 2 \left ( \frac{\partial \beta}{\partial \xi} \right )^2 + e^{2\xi}h(e^\xi).
$$
Now observe that Propositions \ref{latitude_control} and \ref{latitude_longitude}, along with our change of variables, imply that $\sin^2(\beta) \in L^1(H_{-\lambda})$, whereas $\nabla \alpha_1, \nabla \beta \in L^2(H_{-\lambda})$ by Step 1.  Since $\cos^2(\beta) \geq \frac{1}{2}$ in $H_{-\lambda}$, this concludes the proof of Step 3.
\end{proof}

\medskip
\medskip

\noindent {\bf Step 4.}  There is a constant $\alpha^* \in \R{}$ such that
$$
\alpha_1(\xi_1, \xi_2) \to  \alpha^* \,\,\,\mbox{as}\,\,\,  \xi_1 \to -\infty,
$$
uniformly in $\xi_2\in \R{}$.

\begin{proof}[Proof of Step 4.]  Let
$$
\overline{\alpha}_1(\xi_1) = \frac{1}{2\pi} \int_{-\pi}^\pi \alpha_1(\xi_1, \xi_2)\,d\xi_2,
$$
and observe that, for $\xi_{1, 1} < \xi_{1, 2} < -\lambda$, we have
$$
\overline{\alpha}_1(\xi_{1, 2}) - \overline{\alpha}_1(\xi_{1, 1}) = \frac{1}{2\pi} \int_{\xi_{1, 1}}^{\xi_{1, 2}} \left ( \int_{-\pi}^\pi \frac{\partial \alpha_1}{\partial \xi_1}(s, t)\,dt\right ) \,ds.
$$
Hence
$$
\abs{\overline{\alpha}_1(\xi_{1, 2}) - \overline{\alpha}_1(\xi_{1, 1})} \leq \frac{1}{2\pi} \int_{\xi_{1, 1}}^{\xi_{1, 2}} \int_{-\pi}^\pi \abs{\nabla \alpha_1}(s, t)\,dt \,ds.
$$
By Step 3, $\overline{\alpha}_1(\xi_1)$ is Cauchy as $\xi_1\to -\infty$.  Since $(\nabla \alpha_1)(\xi_1, \xi_2)\to 0$ as $\xi_1\to -\infty$, uniformly in $\xi_2\in \R{}$, this proves Step 4.

\end{proof}

\medskip
\medskip

\noindent {\bf Step 5.}  We have the lower bound
\begin{align*}
\int_{B_r(a) } e_\eps(u_\eps) &\geq I(r, \eps) - o(1) \\ &- C\int_{B_r(a)}\left ( \abs{\nabla \alpha_1}^2 + \abs{\nabla \beta}^2 + \frac{\abs{\nabla \alpha_1}}{\abs{x-a}} \right ) - \frac{C}{r^2}\int_{B_r(a)} (\alpha_1^2+\beta^2),
\end{align*}
where $o(1) \to 0$ as $\eps \to 0$.

\begin{proof}[Proof of Step 5.]  In this step we will work in $B_r(a)$.  Then, by Step 4, we can apply a fixed rotation to $u$ so as to obtain $\alpha_1(x) \to 0$ as $x\to a$.

\medskip
\medskip

To prove the claim in this step we build a suitable comparison map.  Recall that $v_\eps$ denotes the nearest point projection of $u_\eps$ on $\mathcal P$.  We will denote $\Pi(\omega)$ the nearest point projection of $\omega \in M^3_{s, 1}(\R{})$ onto $\mathcal P$, whenever this projection is well defined and unique.  Our comparison map is
$$
\omega_\eps(x) = \left \{  \begin{array}{cc}  u_\eps(x) & \mbox{for}\,\,\, \abs{x-a}< \frac{r}{2} \\  
 \left ( 4 - \frac{6\abs{x-a}}{r}\right ) u_\eps + \left ( \frac{6\abs{x-a}}{r} - 3 \right ) v_\eps & \mbox{for}\,\,\,\frac{r}{2} < \abs{x-a} < \frac{2r}{3} \\
 \Pi\left ( \left ( 5 - \frac{6\abs{x-a}}{r}\right ) v_\eps + \left ( \frac{6\abs{x-a}}{r} - 4 \right ) u \right ) & \mbox{for}\,\,\,\frac{2r}{3} < \abs{x-a} < \frac{5r}{6} \\
u\left ( \theta + \frac{6(r-\abs{x-a})}{r} \alpha_1; \frac{6(r-\abs{x-a})}{r} \beta \right ) &\mbox{for}\,\,\, \frac{5r}{6} < \abs{x-a} < r. 
 \end{array} \right .
$$
Note that $\omega_\eps$ has canonical flat data on $\partial B_r(a)$.  Hence
$$
\int_{B_r(a)} e_\eps(\omega_\eps) \geq I(r, \eps).
$$
We now estimate $\int_{B_r(a)} e_\eps(\omega_\eps)$ in each of the intervals for $\abs{x-a}$ that appear in the definition of $\omega_\eps$.

\medskip
\medskip

\noindent In the range $\abs{x-a} < \frac{r}{2}$, clearly we have
$$
\int_{B_{\frac{r}{2}}(a)} e_\eps(\omega_\eps) = \int_{B_{\frac{r}{2}}(a)} e_\eps(u_\eps).
$$

\medskip
\medskip

\noindent Let now $\frac{r}{2} < \abs{x-a} < \frac{2r}{3}$.  In this case we first observe that
$$
{\rm dist}(\omega_\eps, {\mathcal P}) \leq {\rm dist}(u_\eps, {\mathcal P}).
$$
Since 
$$
4W_\beta(\omega_\eps) \leq (1-\abs{\omega_\eps}^2)^2 \leq ({\rm dist}(\omega_\eps, {\mathcal P}))^2 \leq ({\rm dist}(u_\eps, {\mathcal P}))^2 \leq C W_\beta(u_\eps), 
$$
we obtain
$$
\int_{B_{\frac{2r}{3}}(a) \setminus B_{\frac{r}{2}}(a)} \frac{W_\beta(\omega_\eps)}{\eps^2} = o(1).
$$
This last claim follows from the end of the proof of Lemma 8 of \cite{GM}, which shows that minimizers $u_\eps$ satisfy
$$
\limsup_{\eps\to 0} \int_{\Omega\setminus B_r(a)} \frac{W(u_\eps)}{\eps^2} = 0.
$$

\medskip
\medskip

Next, from \cite{GM} we also know that
$$
\int_{B_{\frac{2r}{3}}(a) \setminus B_{\frac{r}{2}}(a)} \abs{\nabla v_\eps}^2 \leq \int_{B_{\frac{2r}{3}}(a) \setminus B_{\frac{r}{2}}(a)}\abs{\nabla u_\eps}^2 + o(1).
$$
All this gives us
$$
\int_{B_{\frac{2r}{3}}(a) \setminus B_{\frac{r}{2}}(a)} e_\eps(\omega_\eps) \leq \int_{B_{\frac{2r}{3}}(a) \setminus B_{\frac{r}{2}}(a)} e_\eps(u_\eps) + o(1).
$$

\medskip
\medskip

\noindent For the range $\frac{2r}{3} < \abs{x-a} < \frac{5r}{6}$, we first observe that
$$
\Pi\left ( \left ( 5 - \frac{6\abs{x-a}}{r}\right ) v_\eps + \left ( \frac{6\abs{x-a}}{r} - 4 \right ) u \right ) = \Pi\left ( u + \left ( 5 - \frac{6\abs{x-a}}{r}\right ) (v_\eps - u) \right ).
$$
Set
$$
z_{\eps, r} = u + \left ( 5 - \frac{6\abs{x-a}}{r}\right ) (v_\eps - u),
$$
so that $\omega_\eps= \Pi(z_{\eps.r})$.  We have
\begin{align*}
\frac{\partial \omega_\eps}{\partial x_k} &= (D\Pi)(z_{\eps, r})(\frac{\partial z_{\eps, r}}{\partial x_k}) \\ &= \frac{\partial u}{\partial x_k} + ((D\Pi)(z_{\eps, r}) - (D\Pi)(u)) ( \frac{\partial u}{\partial x_k} ) + (D\Pi)(z_{\eps, r})(\frac{\partial z_{\eps, r}}{\partial x_k} - \frac{\partial u}{\partial x_k}).
\end{align*}
Because $u_\eps \to u$ and $v_\eps \to u$ strongly in $B_r(a) \setminus B_{\frac{r}{2}}(a)$, using the facts described in the appendix we obtain
$$
\int_{B_{\frac{5r}{6}}(a) \setminus B_{\frac{2r}{3}}(a) } \abs{\nabla \omega_\eps}^2 \leq \int_{B_{\frac{5r}{6}}(a) \setminus B_{\frac{2r}{3}}(a) } \abs{\nabla u}^2 + o(1) \leq \int_{B_{\frac{5r}{6}}(a) \setminus B_{\frac{2r}{3}}(a) } \abs{\nabla u_\eps}^2 + o(1).
$$
Since $W_\beta(\omega_\eps) = 0$, all this gives us
$$
\int_{B_{\frac{5r}{6}}(a) \setminus B_{\frac{2r}{3}}(a) } e_\eps(\omega_\eps) \leq \int_{B_{\frac{5r}{6}}(a) \setminus B_{\frac{2r}{3}}(a) } e_\eps(u_\eps) + o(1).
$$

\medskip
\medskip

\noindent Finally, let $\frac{5r}{6} < \abs{x-a} < r$.  We observe that
\begin{align*}
\nabla \omega_\eps &= \left ( \nabla \theta + \frac{6(r-\abs{x-a})}{r} \nabla \alpha_1 - \frac{6}{r} \alpha_1 \,\hat{n} \right ) \frac{\partial u}{\partial \alpha} \\
&+ \left ( \frac{6(r-\abs{x-a})}{r} \nabla \beta - \frac{6}{r} \beta \,\hat{n} \right ) \frac{\partial u}{\partial \beta},
\end{align*}
where we use the notation $\hat{n} = \frac{x-a}{\abs{x-a}}$.  From this identity we obtain
\begin{align*}
\nabla \omega_\eps &=  \nabla u + \left ( \frac{(5r-6\abs{x-a})}{r} \nabla \alpha_1 - \frac{6}{r} \alpha_1 \,\hat{n} \right ) \frac{\partial u}{\partial \alpha} \\
&+ \left ( \frac{(5r-6\abs{x-a})}{r} \nabla \beta - \frac{6}{r} \beta \,\hat{n} \right ) \frac{\partial u}{\partial \beta}.
\end{align*}
Since $\langle \frac{\partial u}{\partial \alpha}, \frac{\partial u}{\partial \beta} \rangle = 0 $, it follows that
$$
\abs{\nabla \omega_\eps}^2 \leq \abs{\nabla u}^2 + C\int_{B_r(a)}\left ( \abs{\nabla \alpha_1}^2 + \abs{\nabla \beta}^2 + \frac{\abs{\nabla \alpha_1}}{\abs{x-a}} \right ) + \frac{C}{r^2}\int_{B_r(a)} (\alpha_1^2+\beta^2).
$$
Since $W_\beta(\omega_\eps) = 0$, we finally get
\begin{align*}
\int_{B_r(a) \setminus B_{\frac{5r}{6}}(a) } e_\eps(\omega_\eps) &\leq \int_{B_r(a) \setminus B_{\frac{5r}{6}}(a) } e_\eps(u_\eps) \\ &+ C\int_{B_r(a)}\left ( \abs{\nabla \alpha_1}^2 + \abs{\nabla \beta}^2 + \frac{\abs{\nabla \alpha_1}}{\abs{x-a}} \right ) + \frac{C}{r^2}\int_{B_r(a)} (\alpha_1^2+\beta^2).
\end{align*}
Putting together the estimates in the various ranges for $\abs{x-a}$, we conclude the proof of this step.
\end{proof}

\noindent {\bf Step 6.}  We have the upper bound
\begin{align*}
\int_{B_r(a)} e_\eps(u_\eps) &\leq I(r, \eps) + o(1)  \\ &+ C\int_{B_r(a)}\left ( \abs{\nabla \alpha_1}^2 + \abs{\nabla \beta}^2 + \frac{\abs{\nabla \alpha_1}}{\abs{x-a}} \right ) + \frac{C}{r^2}\int_{B_r(a)} (\alpha_1^2+\beta^2).
\end{align*}

\begin{proof}[Proof of Step 6.]  Let $\zeta_\eps$ be a minimizer of LdG in $B_r(a)$ with canonical flat data, and define
$$
\omega_\eps(x) = \left \{  \begin{array}{cc} \zeta_\eps & \abs{x-a} < r \\ u\left ( \theta + \left ( \frac{3\abs{x-a}}{r} - 3 \right ) \alpha_1; \left ( \frac{3\abs{x-a}}{r} -3\right ) \beta \right ) & \mbox{for}\,\,\, r < \abs{x-a}< \frac{4r}{3} \\  
 \Pi\left ( \left ( 5 - \frac{3\abs{x-a}}{r}\right ) u + \left ( \frac{3\abs{x-a}}{r} - 4 \right ) v_\eps \right ) & \mbox{for}\,\,\,\frac{4r}{3} < \abs{x-a} < \frac{5r}{3} \\
\left ( 6 - \frac{3\abs{x-a}}{r}\right ) v_\eps + \left ( \frac{3\abs{x-a}}{r} - 5 \right ) u_\eps &\mbox{for}\,\,\, \frac{5r}{3} < \abs{x-a} < 2r. 
 \end{array} \right .
$$
Note that $\zeta_\eps = u_\eps$ on $\partial B_{2r}(a)$.  Since $u_\eps$ minimizes the LdG energy with respect to its own boundary data, we obtain
$$
\int_{B_{2r}(a)} e_\eps(\zeta_\eps) \geq \int_{B_{2r}(a)} e_\eps(u_\eps).
$$
Furthermore, by definition we have
$$
\int_{B_{r}(a)} e_\eps(\zeta_\eps) = I(r, \eps).
$$
To conclude we follow essentially the same strategy we used in Step 5 to show that
\begin{align*}
\int_{B_{2r}(a)\setminus B_{r}(a)} e_\eps(\zeta_\eps) &\leq \int_{B_{2r \setminus B_{r}(a)}(a)} e_\eps(u_\eps) + o(1) \\
&+C\int_{B_{2r}(a)}\left ( \abs{\nabla \alpha_1}^2 + \abs{\nabla \beta}^2 + \frac{\abs{\nabla \alpha_1}}{\abs{x-a}} \right ) + \frac{C}{r^2}\int_{B_{2r}(a)} (\alpha_1^2+\beta^2).
\end{align*}
This concludes the proof of the theorem.
\end{proof}\renewcommand{\qedsymbol}{}
\end{proof}

\section{Estimates away from the singularity}

In this section we prove Theorem \ref{current_vector}.  Recall that we are assuming $a=0$, and denote the zero set of the Hopf differential $\omega_u$ of $u$ by
$$
Z_{\omega_u}= \{z\in \Omega\setminus \{a\} : \omega_0(z) = 0\}.
$$
From Proposition \ref{Hopf_Diff_split} we have
$$
\omega_u = -\frac{1}{8z^2} + h,
$$
where $h$ is holomorphic in $\Omega$.  Our main result in this section gives an expression for $j(u)$ in the case of $Z_{\omega_u}=\emptyset$.  In this situation the function $1+8z^2h$ does not vanish in $\Omega$.  Because $\Omega$ is simply-connected, we can extract a square root of $1+8z^2h_0$, and hence of $\omega_0$.  For convenience, let $\mu_u$ be a (necessarily meromorphic) function that satisfies $-2\mu_u^2=\omega_u$.  We observe, however, that when $Z_{\omega_u}\neq \emptyset$, the conclusions are still valid, but only locally away from $Z_{\omega_u}$.

\begin{proof}[Proof of Theorem \ref{current_vector}.]

Consider the exponential map $e : \C{}\to \C{}\setminus \{0\}$, and let
\begin{equation}\label{def_H}
H_{\Omega^*} = e^{-1}(\Omega \setminus \{0\}).
\end{equation}
In other words, $H_{\Omega^*}$ is the lift of the punctured domain $\Omega\setminus \{0\}$.  It is well known that $H_{\Omega^*}$ along with the (complex) exponential map is the universal covering space of $\Omega\setminus \{0\}$.  Any map $u:\Omega\setminus\{0\}\to X$, into any topological space, defines a map $v:H_{\Omega^*} \to X$ via $v(\xi)=u(e^\xi)$.  Observe that $v$ is $2\pi i$-periodic in $H_{\Omega^*}$.  On the other hand, any map $v:H_{\Omega^*}\to X$, into the topological space $X$, that is $2\pi i$-periodic, induces a unique map $u:\Omega\setminus\{0\}\to X$ such that $v(\xi)=u(e^\xi)$.  We will abuse the notation and say that a map $v:H_{\Omega}\to X$, that is {\it not} $2\pi i$-periodic, is a multivalued map in $\Omega\setminus \{0\}$.  

Next, for a fixed $P \in \mathcal P$, we define
$$
Q_P(A) = AP+PA
$$
for any $A\in M^3(\R{})$.  It turns out that
$$
Q_P:M_a^3(\R{})\to M_a^3(\R{})
$$
is an orthogonal projection with respect to the inner product $\langle A, B \rangle = {\rm tr}(B^TA)$ for $A, B \in M^3_a(\R{})$.  Denoting the image of this projection by $A_P$, it is easy to check that
$$
Q_P:T_P {\mathcal P} \to A_P
$$
is an isomorphism.

\medskip
\medskip

Let now $u:\Omega\setminus \{a\} \to \mathcal P$ be a limit of minimizers and lift it to $H_{\Omega^*}$ through $v(\xi)=u(e^\xi)$.  Recall that $v$ is $2\pi i$-periodic in $H_{\Omega^*}$.  As before, we define
$$
j(v)=\left [v; \frac{\partial v}{\partial \xi} \right ].
$$
We now recall that $u$ satisfies a minimality condition.  Since the change of variables to go from $\Omega\setminus \{a\}$ to $H_{\Omega^*}$ is conformal, $v$ also satisfies a minimality condition so that
\begin{equation}\label{laplcian_v_commute}
[v; \Delta v]=0.
\end{equation}
This, plus the fact that $v$ is $\mathcal P$-valued, implies that
\begin{equation}\label{current_derivative}
\frac{\partial j(v)}{\partial \bar{\xi}} = -\left [ \overline{j(v)}; j(v)\right ].
\end{equation}
We next recall that the map $L_P : SO(3)\to \mathcal P$, where $P \in \mathcal P$ is fixed, defined through
$$
L_p(R) = RPR^T,
$$
is onto.  It is not, however, even locally injective.  In fact, its stabiliser
$$
O_P(3) = \{R\in SO(3): RPR^T=P\},
$$
is non-trivial.  Although $L_P$ is not a covering map, we can still lift $v$ with some $R:H_{\Omega^*} \to SO(3)$, so that
$$
v = RPR^T.
$$
This can be seen by building the map $R$ locally around any $v(\xi)$, $\xi \in H_{\Omega^*}$, and extending it.

\medskip
\medskip

The lifting $R$ need not be unique.  However, since $v(\xi + 2\pi i)=v(\xi)$, we must have
$$
R^T(\xi)R(\xi+2\pi i)\in O_P(3)
$$
for all $\xi \in H_{\Omega^*}$.  Next observe that, since $P$ is constant,
$$
\frac{\partial v}{\partial \xi} = \frac{\partial R}{\partial \xi}PR^T + RP\frac{\partial R^T}{\partial \xi}.
$$
Now $R$ takes values in $SO(3)$.  Hence
$$
{\rm Re}\left (R^T\frac{\partial R}{\partial \xi}\right ), {\rm Im}\left (R^T\frac{\partial R}{\partial \xi}\right ) \in M_a^3(\R{}).
$$
Since $P \in \mathcal P$, we obtain
$$
j(v) = R\left ( P\frac{\partial R^T}{\partial \xi}R - R^T\frac{\partial R}{\partial \xi}P \right )R^T.
$$
Denote
$$
B(\xi) = \frac{\partial R^T}{\partial \xi}R.
$$
If we set
$$
\beta(\xi) = Q_P(B),
$$
then so far we only have that
$$
j(v) = R\beta R^T.
$$
We will prove next that, in fact,
$$
\frac{\partial \beta}{\partial \bar{\xi}} = \left [ \overline{B-\beta}; \beta\right ].
$$
To this end, we observe that $R:H_{\Omega^*}\to O(3)$ has real entries.  Hence
$$
\overline{\frac{\partial R}{\partial \xi}} = \frac{\partial R}{\partial \bar{\xi}}.
$$
From here we obtain
$$
\frac{\partial j(v)}{\partial \bar{\xi}} = R\left ( \frac{\partial \beta}{\partial \bar{\xi}} + R^T\frac{\partial R}{\partial \bar{\xi}}\beta + \beta \frac{\partial R^T}{\partial \bar{\xi}} R\right )R^T = R\left ( \frac{\partial \beta}{\partial \bar{\xi}} -[\bar{B}; \beta]\right )R^T,
$$
where again we used the fact that $R$ has real entries.  Next, we already know that
$$
\frac{\partial j(v)}{\partial \bar{\xi}} = \left [ \frac{\partial v}{\partial \bar{\xi}}; \frac{\partial v}{\partial \xi}\right ] = -\left [ \overline{j(v)}; j(v)\right ] = -R[\bar{\beta}; \beta]R^T.
$$
This proves
$$
\frac{\partial \beta}{\partial \bar{\xi}} = \left [ \overline{B-\beta}; \beta\right ].
$$

\medskip
\medskip

Let now $\Lambda_1 \in M_A^3(\R{})$ be (constant and) such that $Q_P(\Lambda_1)=0$ and
$$
[\Lambda_1; [\Lambda_1; A]] = -Q_P(A) \,\,\, \mbox{for all}\,\,\, A\in M_A^3(\R{}).
$$
Let also $\Lambda_2, \Lambda_3 \in A_P$ be such that $\{\Lambda_1, \Lambda_2, \Lambda_3\}$ is an orthogonal basis in $M_a^3(\R{})$ with the additional property
$$
[\Lambda_1, \Lambda_2]=\Lambda_3, [\Lambda_2, \Lambda_3]=\Lambda_1 \,\,\,\mbox{and}\,\,\,[\Lambda_3, \Lambda_1]=\Lambda_2.
$$
By the definition of $\beta$, there is a function $\alpha_1 : H_{\Omega^*}\to \C{}$ such that
$$
B(\xi) = \alpha_1 \Lambda_1 + \beta.
$$
What we know so far can be expressed as
$$
\frac{\partial \beta}{\partial \bar{\xi}} = \overline{\alpha}_1T_1( \beta),
$$
where $T_1(B) = T_{\Lambda_1}(B) = [\Lambda_1; B]$.

\medskip
\medskip

Let now $a_1 : H_{\Omega^*}\to \C{}$ be such that
$$
\frac{\partial a_1}{\partial \bar{\xi}}=\bar{\alpha_1}.
$$
To see that such $a_1$ should exist, we write $\alpha_1 = s + it$, and seek real-valued functions $f,g$ such that
$$
\frac{\partial ^2}{\partial \xi\, \partial \bar{\xi}} f = s, \,\,\, \frac{\partial ^2}{\partial \xi\, \partial \bar{\xi}} g = t.
$$
Such $f$ and $g$ always exist in a half-plane by Theorem 3.6.4 in \cite{Horm}.  Since $H_{\Omega^*}$ is conformal to a half space, $f$ and $g$ also exist in $H_{\Omega^*}$.  With these functions, we set
$$
a_1 = \frac{\partial f}{\partial \xi} - i\frac{\partial g}{\partial \xi}.
$$
By construction,
$$
\frac{\partial a_1}{\partial \bar{\xi}}=\overline{\alpha}_1.
$$
But then
$$
\frac{\partial \beta}{\partial \bar{\xi}} = \frac{\partial a_1}{\partial \bar{\xi}}[\Lambda_1,  \beta],
$$
and we deduce that
$$
\frac{\partial}{\partial \bar{\xi}} \left (   e^{-a_1\Lambda_1} \beta  e^{a_1\Lambda_1} \right ) = 0,
$$
where $e^{a_1\Lambda_1}$ is the standard exponential of a matrix.  Because of the definition of $\Lambda_2$, $\Lambda_3$, there are holomorphic functions $z_2, z_3 : H_{\Omega^*} \to \C{}$ such that
\begin{equation}\label{beta_first_expr}
\beta = e^{a_1 \Lambda_1 }(z_2\Lambda_2 + z_3 \Lambda_3) e^{-a_1\Lambda_1}.
\end{equation}
Define now
\begin{equation}\label{def_omega_v}
\omega_v(\xi) = \langle j(v) , j(v)\rangle.
\end{equation}
One checks directly that
$$
\omega_v = 2(z_1^2 + z_2^2).
$$

\medskip
\medskip

Let now $\mu_v$ be a holomorphic map such that $2\mu_v^2 = \omega_v$.  Lemma \ref{lemmita} allows us then to find a holomorphic $\zeta$ such that
$$
z_2 = \mu_v \cos(\zeta), \,\,\,  z_3 = \mu_v \sin(\zeta).
$$
This, in particular, shows that
\begin{equation}
    \label{eq:22}
    z_2\Lambda_2 + z_3 \Lambda_3 = \mu_v e^{\zeta T_1}(\Lambda_2),
\end{equation}
where $T_1(A) = [\Lambda_1, A]$.  Along with \eqref{beta_first_expr} the equation \eqref{eq:22} 
gives
\begin{equation}\label{beta_second_expr}
\beta = \mu_v e^{(a_1+\zeta)T_1}(\Lambda_2).
\end{equation}

\medskip
\medskip

We now observe the following: if $f:H_{\Omega^*}\to \R{}$ is any function, then
$$
S=Re^{f\Lambda_1} \in O(3)
$$
also satisfies
$$
v = SPS^T.
$$
Furthermore, setting
$$
B_S(\xi) = \frac{\partial S^T}{\partial \xi}S,
$$
it is not hard to see that
$$
\beta_S = Q_P(B_S) = S^TR\beta R^TS = e^{-f\Lambda_1}\beta e^{f\Lambda_1} = e^{-fT_1}(\beta),
$$
and $e^{f\Lambda_1}$ is the exponential of a matrix, whereas $e^{fT_1}$ is the exponential of an operator in $M_a^3(\R{})$ (which incidentally can also be written as a $3\times 3$ matrix with respect to the appropriate basis in $M_a^3(\R{})$).  This and \eqref{beta_second_expr} yield
$$
\beta_S = \mu_v e^{(a_1+\zeta-f)T_1}(\Lambda_2).
$$
Since the function $f$ is arbitrary, except for the fact that it must be real-valued, we set 
$$
f={\rm Re}(a_1+\zeta), \,\,\,\,   g={\rm Im}(a_1+\zeta).
$$
We conclude that
$$
\beta_S = \mu_v e^{igT_1}(\Lambda_2) = \mu_v ( \cosh(g) \Lambda_2 + i\sinh(g)\Lambda_3).
$$
For the final conclusion we notice that a direct computation shows that
$$
B_S = -i\frac{\partial g}{\partial \xi}\Lambda_1 + \beta_S.
$$
Recall now that
$$
B_S = \frac{\partial S^T}{\partial \xi}S,
$$
and that $S$ has real entries.  Because of this
$$
{\rm Im}\left (  \frac{\partial^2 S^T}{\partial \xi\, \partial \bar{\xi}} \right ) = 0.
$$
Since
$$
\frac{\partial S^T}{\partial \xi} = B_SS^T,
$$
we obtain from here that
$$
{\rm Im}\left ( \left (\frac{\partial B_S}{\partial \bar{\xi}} + B_S\overline{B_S}\right ) S^T\right ) = 0.
$$
Observing that
$$
B_S\overline{B_S} = \frac{1}{2}[B_s; \overline{B_S}] + \frac{1}{2}(B_S\overline{B_S}+\overline{B_S}B_S),
$$
and
$$
{\rm Im}\left ( \frac{1}{2}(B_S\overline{B_S}+\overline{B_S}B_S)\right ) = 0,
$$
we obtain
$$
{\rm Im}\left (\frac{\partial B_S}{\partial \bar{\xi}} + \frac{1}{2}[B_S;\overline{B_S}]\right ) = 0.
$$
Inserting everything we have obtained so far into this last equation we obtain
$$
-\frac{\partial ^2 g}{\partial \xi\,\partial \bar{\xi}} = \frac{\abs{\mu_v}^2}{2}\sinh(2g) = \frac{\abs{\omega_v}}{4}\sinh(2g).
$$
Lastlly, we observe that
$$
[\overline{j(v)}, j(v)] = S[\overline{\beta_S}, \beta_S]S^T = \frac{\abs{\omega_v}}{4}\sinh(2g) \Gamma_1,
$$
because $\Gamma_j = S\Lambda_j S^T$.  The conclusions of the theorem now follow by changing variables back from $H_{\Omega^*}$ to $\Omega$.

\end{proof}
We now present the proof of a simple lemma that we used during this proof.
\begin{lemma}\label{lemmita}
Let $D \subset \C{}$ be a simply-connected open set.  For any two holomorphic functions $a$, $b$ in $D$ that satisfy
$$
a^2+b^2 = 1 \,\,\, \mbox{in}\,\,\, D,
$$
there is a holomorphic function $\beta$ in $D$ such that
$$
a = \cos(\beta) \,\,\, \mbox{and}\,\,\, b=\sin(\beta).
$$
\end{lemma}
\begin{proof}
Differentiating $a^2+b^2=1$ we obtain
$$
a\frac{\partial a}{\partial z} + b\frac{\partial b}{\partial z} = 0.
$$
However, $a$ and $b$ cannot be zero simultaneously.  Hence, at least one of the sides of the identity
$$
-\frac{1}{a}\frac{\partial b}{\partial z} = \frac{1}{b}\frac{\partial a}{\partial z}
$$
always makes sense, and is holomorphic.  Set
$$
\alpha= -\frac{1}{a}\frac{\partial b}{\partial z} = \frac{1}{b}\frac{\partial a}{\partial z}.
$$
Observe next that
$$
\frac{\partial}{\partial z}\left ( \begin{array}{c}a\\b\end{array}\right ) = \alpha \left ( \begin{array}{c}-b\\a\end{array}\right ) = \alpha \left ( \begin{array}{cc}0&-1\\1&0 \end{array}\right )\left ( \begin{array}{c}a\\b\end{array}\right ).
$$
Call
$$
T_0=\left ( \begin{array}{cc}0&-1\\1&0 \end{array}\right ),
$$
and let $\beta_1$ be any holomorphic antiderivative of $\alpha$.  What we know so far can be written as
$$
\frac{\partial}{\partial z} \left (e^{-\beta_1T_0}\left ( \begin{array}{c}a\\b\end{array}\right ) \right ) = 0.
$$
Since the expression inside the derivative above is holomorphic, then the expression is constant.  In other words, there are complex constants $c_1, c_2 \in\C{}$ such that
$$
\left ( \begin{array}{c}a\\b\end{array}\right ) = e^{\beta_1T_0}\left ( \begin{array}{c}c_1\\c_2\end{array}\right )= \left ( \begin{array}{c}c_1\cos(\beta_1) - c_2\sin(\beta_1)\\c_1\sin(\beta_1)+c_2\cos(\beta_1)\end{array}\right ).
$$
Observe next that
$$
1=a^2+b^2=c_1^2+c_2^2.
$$
We finish the proof by picking a constant $\beta_2 \in\C{}$ with
$$
c_1 = \cos(\beta_2) \,\,\, \mbox{and}\,\,\, c_2=\sin(\beta_2).
$$
This will imply that
$$
\left ( \begin{array}{c}a\\b\end{array}\right ) = \left ( \begin{array}{c}\cos(\beta_1+\beta_2)\\\sin(\beta_1+\beta_2)\end{array}\right ).
$$
Setting $\beta=\beta_1+\beta_2$ we obtain the conclusion of the Lemma.

To show that we can pick $\beta_2$, we first observe that this is easy to do if either $c_1=0$ or $c_2=0$.  If neither is $0$, let us choose first $\beta_3$ such that $\cos(\beta_3) = c_1$.  To do this recall that
$$
\cos(\beta_3) = \frac{e^{i\beta_3}+e^{-i\beta_3}}{2},
$$
so the equation $c_1=\cos(\beta_3)$
is equivalent to the equation
$$
e^{2i\beta_3}-2c_1e^{i\beta_3}+1=0.
$$
This is a quadratic equation for $\lambda=e^{i\beta_3}$.  Its solutions are
$$
\lambda=\frac{2c_1\pm \sqrt{4c_1^2-4}}{2},
$$
and it is easy to check that neither of these can be $0$.  Since the image of the exponential map is $\C{}\setminus \{0\}$, there always is a $\beta_3$ such that $e^{i\beta_3}=\lambda$.

With this choice of $\beta_3$ we have $c_1=\cos(\beta_3)$.  Now this implies that
$$
1=c_1^2+c_2^2 = \cos^2(\beta_3) + c_2^2 = 1-\sin^2(\beta_3)+c_2^2.
$$
We conclude that either $c_2=\sin(\beta_3)$ or $c_2 = -\sin(\beta_3)$.  In the first case we set $\beta_2=\beta_3$ and we are done.  In the second case we set $\beta_2=-\beta_3$.  Since $\cos(\beta_2) = \cos(-\beta_3)=\cos(\beta_3)$, we are done in this case as well.

\end{proof}

\medskip
\medskip
\medskip
\medskip

\section{Numerical Simulations}
To visualize the results established in the preceding sections, we simulated the gradient flow for the energy functional \eqref{LdGEnergy} using the off-the-shelf finite element analysis solver COMSOL \cite{comsol} in order to arrive at local minimizers of \eqref{LdGEnergy}. The computations were performed in a square domain $\Omega$ with the side of length $1$, assuming that $\varepsilon=0.01$ and using the boundary data of degree $1/2$ with values deviating from a geodesic in $\mathcal P$.

Figs.~\ref{fig:6}-\ref{fig:4} show the eigenvalues and eigenvectors fields for the computed (local) minimizer $u_\varepsilon$ of \eqref{LdGEnergy}. Because the degree of $u_\varepsilon$ on the boundary is equal to $1/2$, there is a single point in $\Omega$ where the eigenvalues of $u_\varepsilon$ should cross and this point should be located near the singularity of the limiting map $u$. To make the subsequent discussion simpler, we will identify the eigenvalues crossover point of $u_\varepsilon$ with the singular point of $u$ in what follows. 

In agreement with Proposition \ref{convergence_geodesic_third_eval}, the third eigenvalue of $u_\varepsilon$ is asymptotically close to $0$, while the corresponding eigenvector field is smooth {\it everywhere} in $\Omega$, including the singular point of $u$. The first and the second eigenvalues of $u_\varepsilon$ are equal to $1$ and $0$, respectively, away from the singularity of $u$, while at the core of the singularity both of these eigenvalues are close to $1/2$.  
\begin{figure}[H]
    \centering
    \includegraphics[scale=.3]{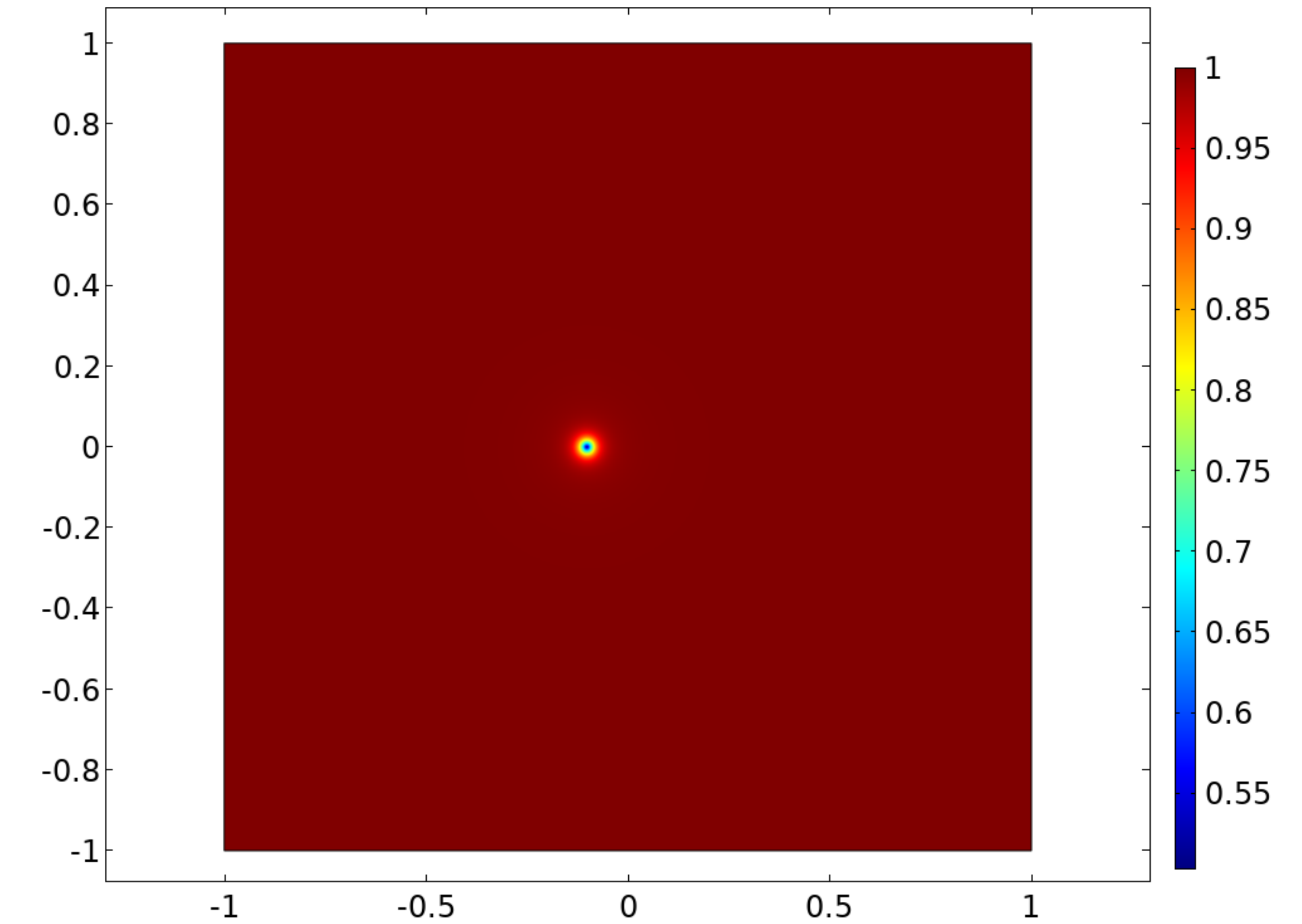}
    \includegraphics[scale=.95]{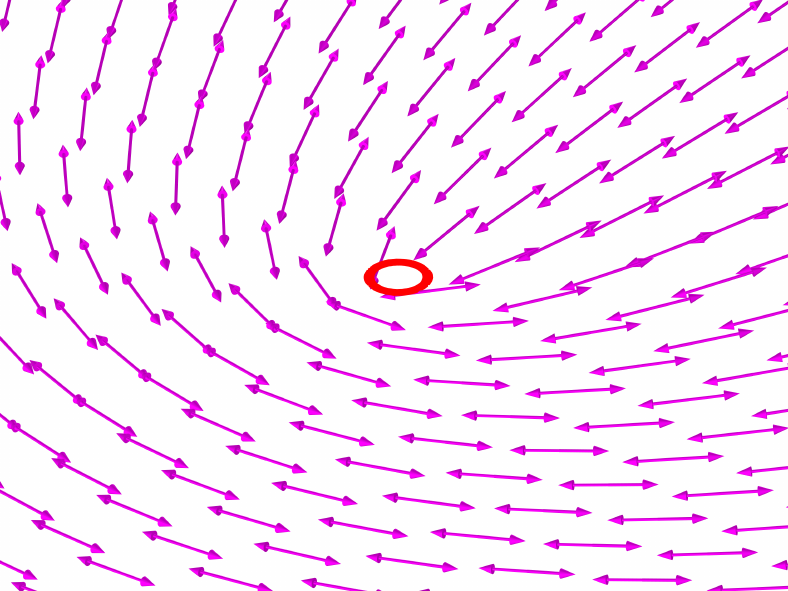}
    \caption{Eigenvalue $\lambda_1$ (left) and eigenvector ${\mathbf e}_1$ (right) of the minimizer $u_\varepsilon$ of \eqref{LdGEnergy}. The vector field plot zooms in on the region near the singularity of $u$, represented by a red circle.}
    \label{fig:6}
\end{figure}
Further, both the first and the second eigenvectors of $u_\varepsilon$  have degree $1/2$ singularity at the singular point of $u$---this is an expected behavior because $u_\varepsilon$ near the singularity is essentially an $\mathbb{RP}^1$-valued map.
\begin{figure}[H]
    \centering
    \includegraphics[scale=.3]{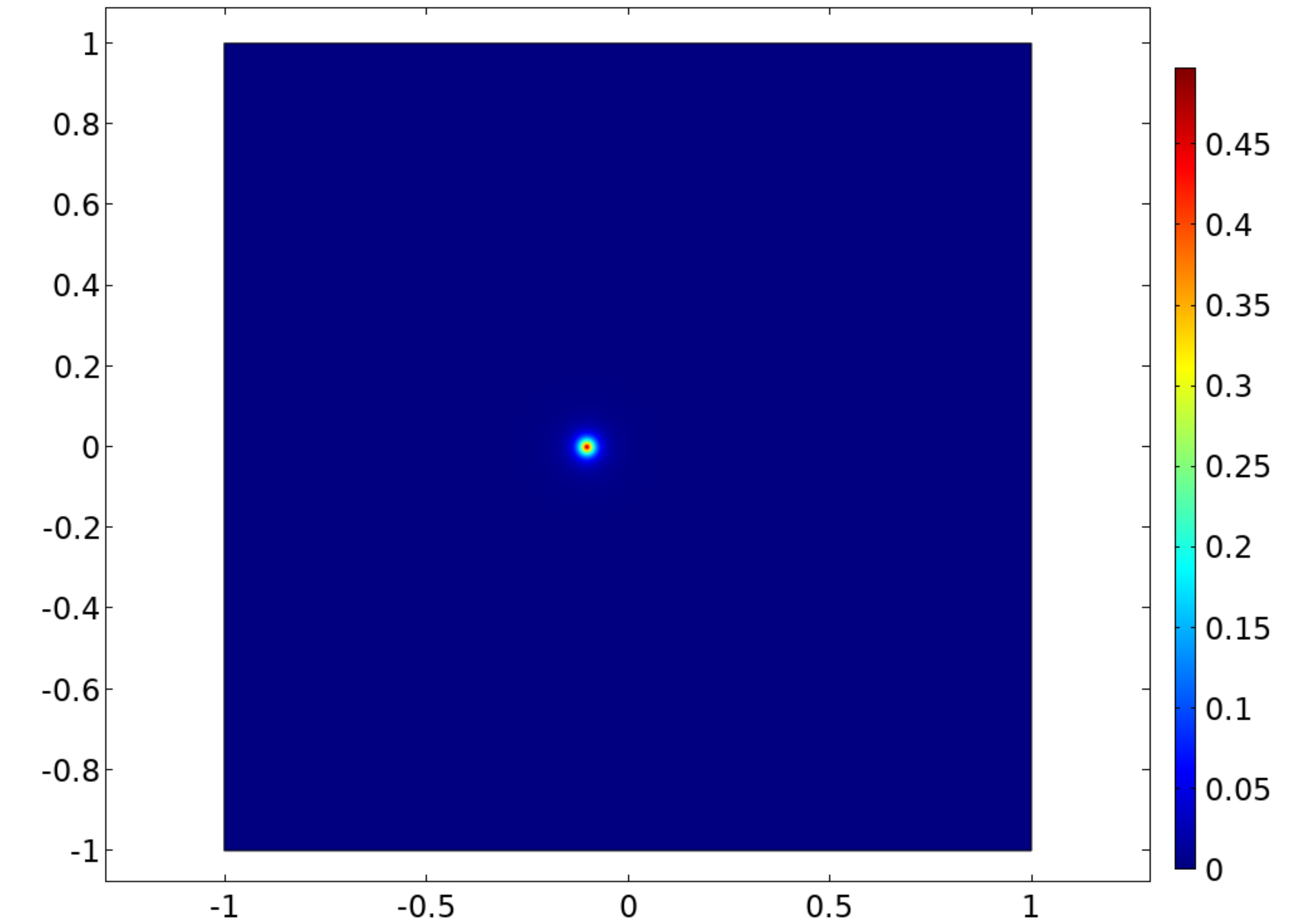}
    \includegraphics[scale=.95]{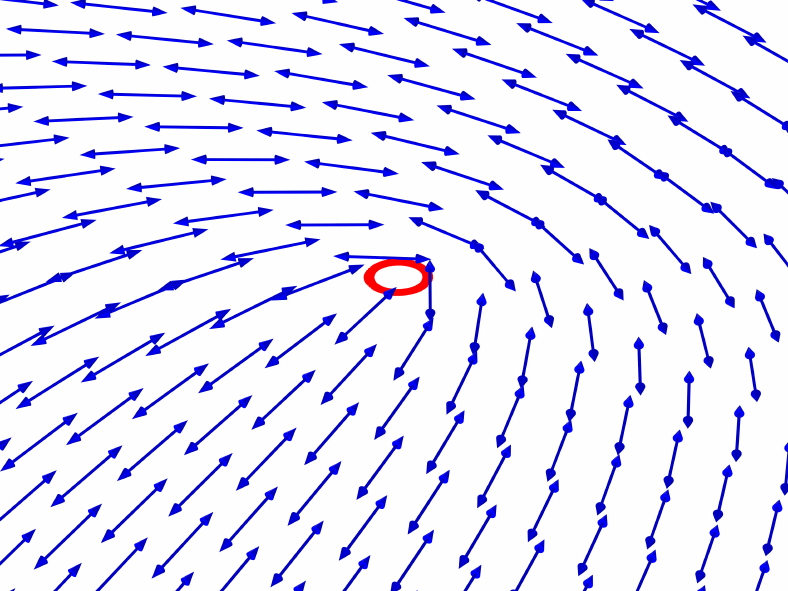}
    \caption{Eigenvalue $\lambda_2$ (left) and eigenvector ${\mathbf e}_2$ (right) of the minimizer $u_\varepsilon$ of \eqref{LdGEnergy}. The vector field plot zooms in on the region near the singularity of $u$, represented by a red circle.}
    \label{fig:5}
\end{figure}
\begin{figure}[H]
    \centering
    \includegraphics[scale=.3]{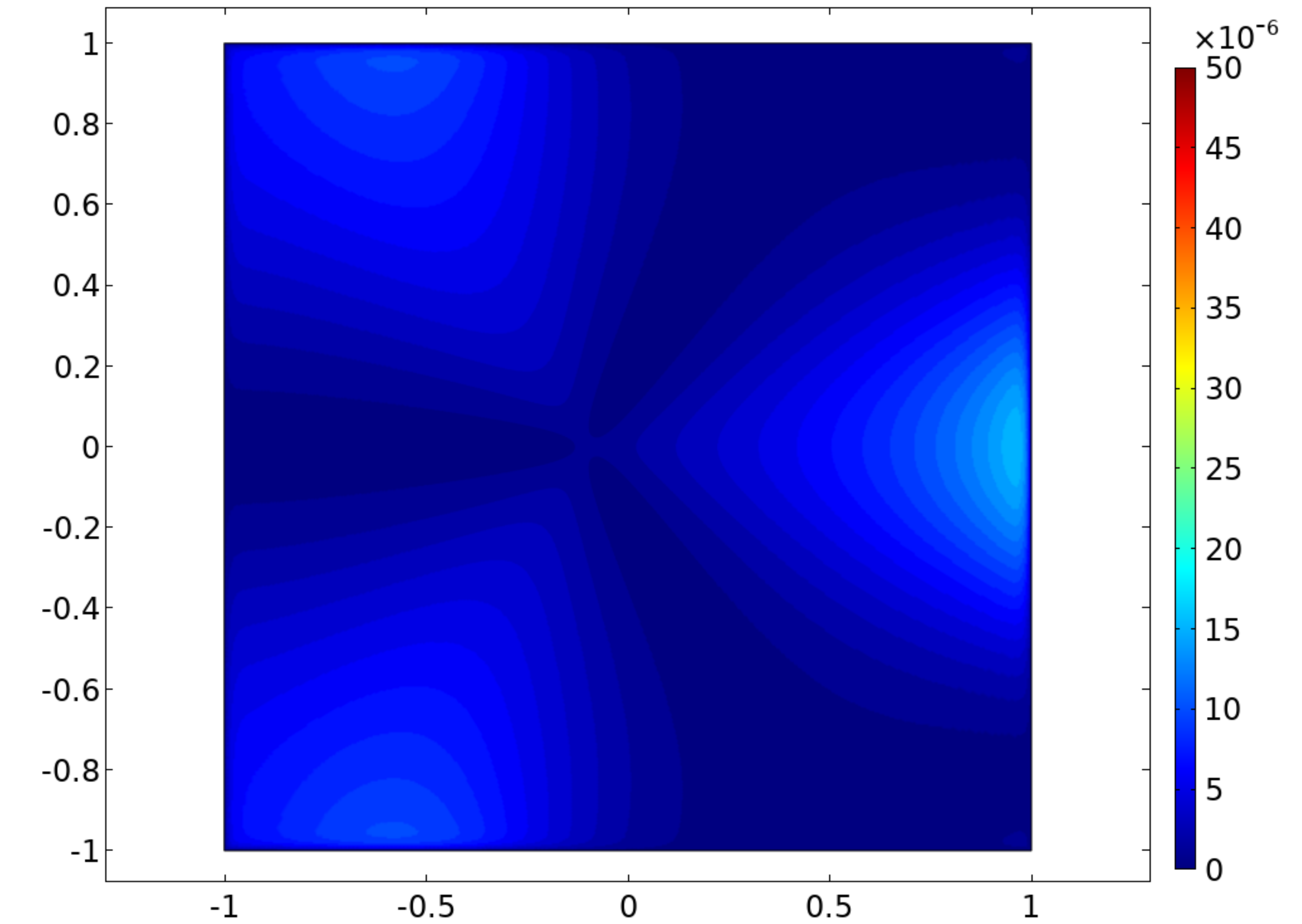}
        \includegraphics[scale=.95]{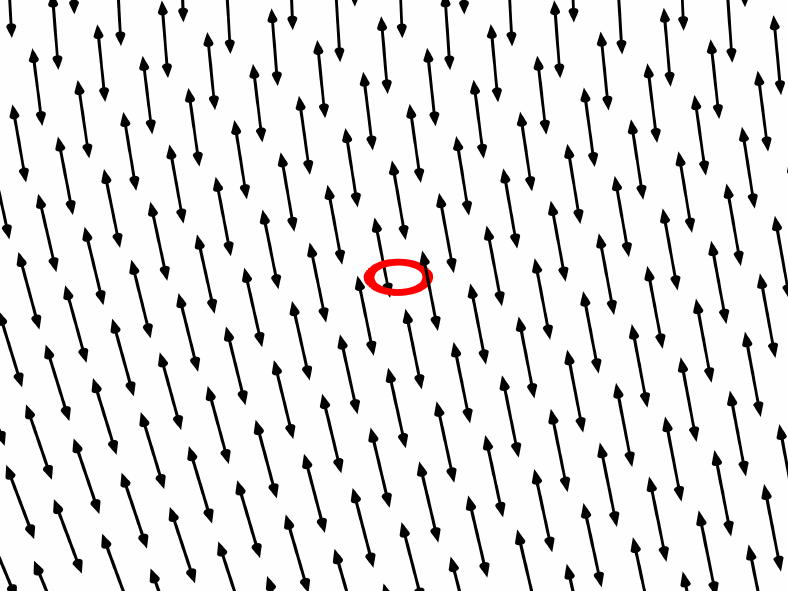}
    \caption{The eigenvalue $\lambda_3$ (left) and the eigenvector ${\mathbf e}_3$ (right) of the minimizer $u_\varepsilon$ of \eqref{LdGEnergy}. The vector field plot zooms in on the region near the singularity of $u$, represented by a red circle.}
    \label{fig:4}
\end{figure}
In Fig.~\ref{fig:10} we plot the distribution of the entire eigenframe of $u_\varepsilon$ in a vicinity of the singular point of $u$.
\begin{figure}[H]
    \centering
    \includegraphics[scale=1.2]{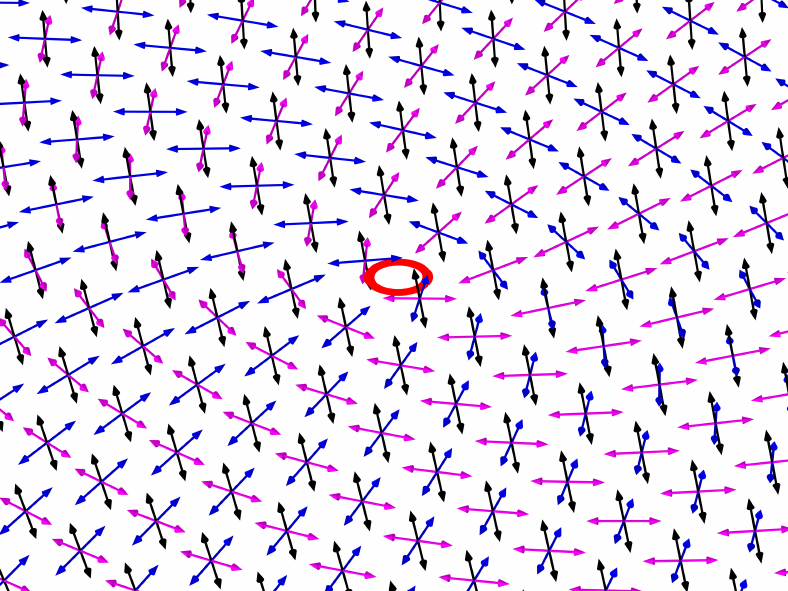}
    \caption{Eigenframe distribution of the minimizer $u_\varepsilon$ of \eqref{LdGEnergy}. The location of the singularity is marked by a red circle.}
    \label{fig:10}
\end{figure}
Figs.~\ref{fig:1}-\ref{fig:3} approximate the behavior of $\mu_u$ that appears in Theorem~\ref{current_vector} as it is computed using $u_\varepsilon$, rather than $u$. From the statement of Theorem~\ref{current_vector}, it follows that $|(z-a)\mu_u|\approx0.25$ near the singularity $a$ of $u$. Indeed, this is what is observed in Fig.~\ref{fig:1}, except that the approximation of $|(z-a)\mu_u|$ plunges to $0$ at $a$, because $\mu_u$ computed using $u_\varepsilon$ instead of $u$ is bounded at $a$. From Theorem~\ref{current_vector}, it also follows that $1/|\mu_u|$ should be linear in the radial coordinate centered at $a$ and this is confirmed by the plot in Fig.~\ref{fig:3}.
\begin{figure}[H]
    \centering
    \includegraphics[scale=.4]{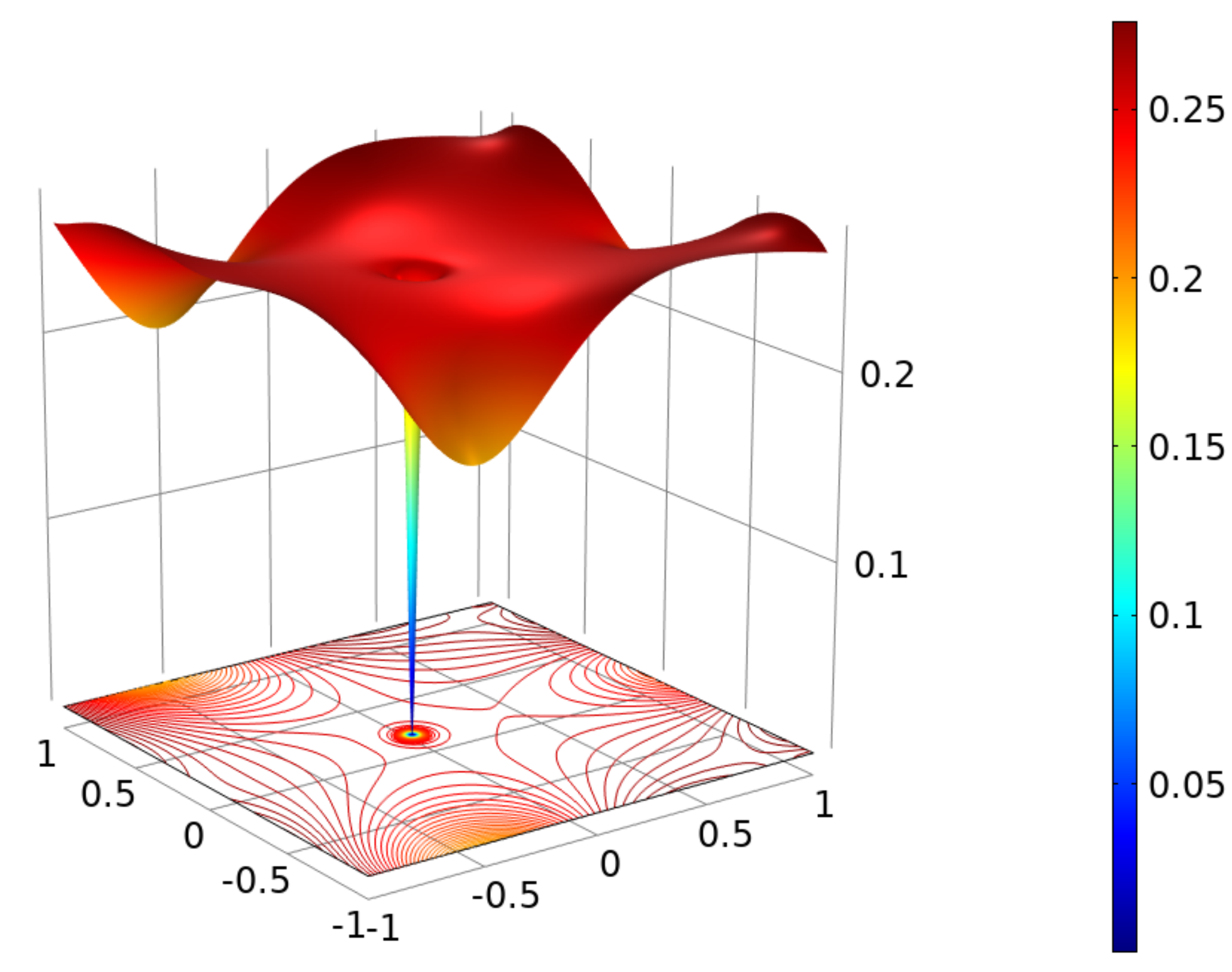}
    \caption{Plot of $|(z-a)\mu_u|$, where $a\in\Omega$ is the location of the singularity of $u$ and $\mu_u$ is as defined in Theorem~\ref{current_vector}.}
    \label{fig:1}
\end{figure}

\begin{figure}[H]
    \centering
    \includegraphics[scale=.4]{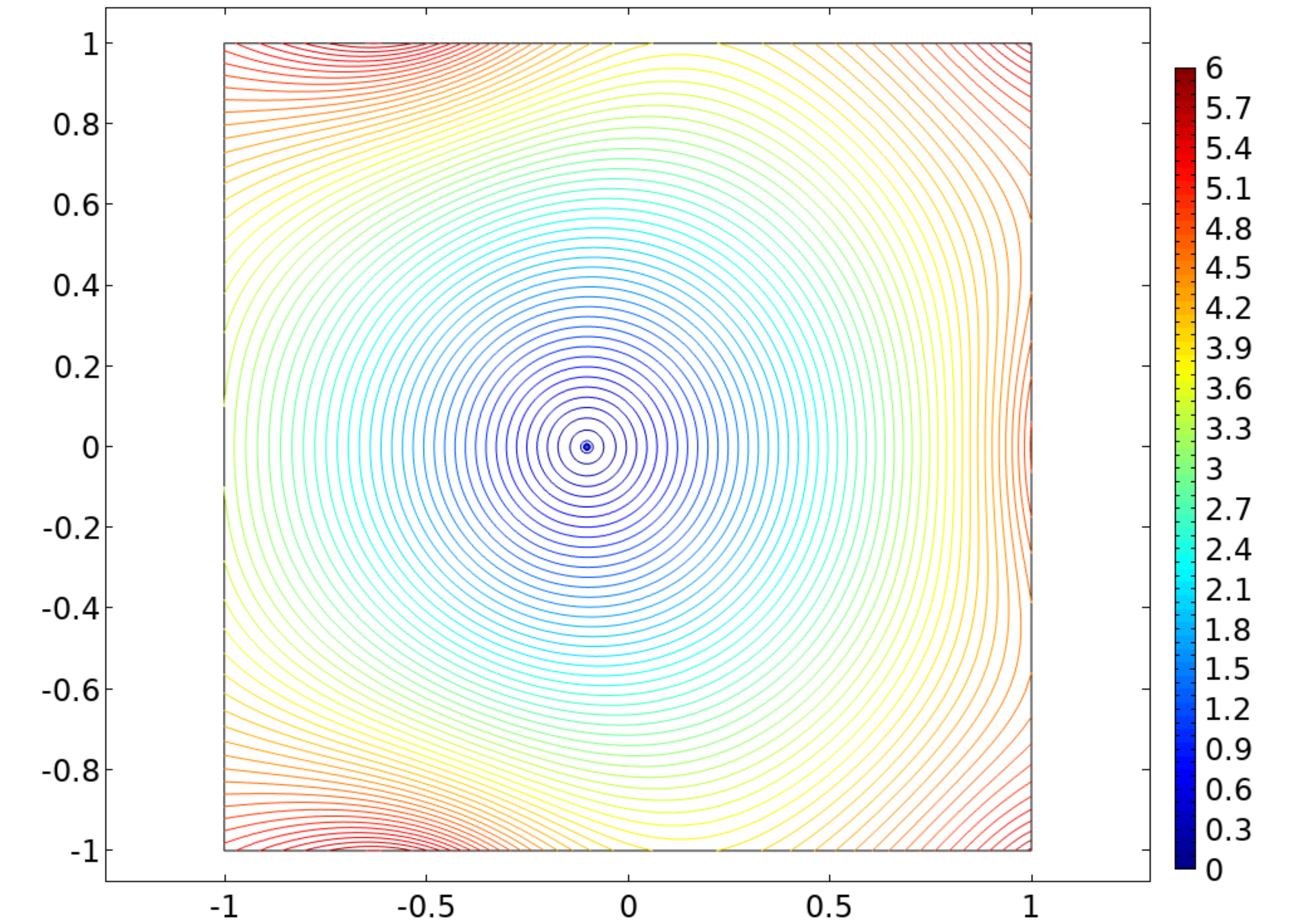}
    \caption{Plot of $1/|\mu_u|$, where $\mu_u$ is as defined in Theorem~\ref{current_vector}.}
    \label{fig:3}
\end{figure}

\section{Acknowledgements}
The first author was supported in part by NSF grant DMS-2106551.

\section{Appendix}

In this appendix we recall a result that was proved in \cite{BBH} for minimizers of the Ginzburg-Landau energy, and that remains valid in our situation.  Let us start by recalling that Cayley-Hamilton theorem for matrices $u\in M_{s, 1}^3(\R{})$ implies
$$
u^3 - u^2 + \frac{(1-\abs{u}^2)}{2}u - {\rm det}(u)I = 0.
$$
Furthermore
\begin{equation}\label{potential}
W_\beta(u) = \frac{1}{4}(1-\abs{u}^2)^2 - \beta \, {\rm det}(u),
\end{equation}
and that for $u \in M^3_{s, 1}(\R{})$, Cayley-Hamilton's theorem gives us
$$
W_\beta(u) = \frac{1}{4}(1-\abs{u}^2)^2 - \frac{\beta}{6}(1 - 3\abs{u}^2 + 2{\rm tr}(u^3)).
$$
From this last expression we obtain
$$
(\nabla W_\beta)(u) = (\abs{u}^2 -1) u + \beta(u-u^2),
$$
and
$$
(D^2_u W_\beta)(u)(h) = (\abs{u}^2-1)h + \beta(h-uh-hu),
$$
also for $M^3_{s, 1}(\R{})$.  Let us also recall that we assume here that $1 \leq \beta < 3$.  In this range we know that $\mathcal P$ is the set of global minimizers of $W_\beta$.  In particular, for any $v\in \mathcal P$ and any $h \in M^3_{s, 0}(\R{})$ we have
$$
\langle (D^2_v W_\beta)(v)(h), h \rangle \geq 0.
$$

\medskip
\medskip

A further consequence of Cayley-Hamilton, for $u \in M_{s, 1}^3(\R{})$, is the identity 
\begin{equation}\label{further_Cayley_Hamilton}
\abs{u-u^2}^2 + 2{\rm det}(u) = \frac{(1-\abs{u}^2)^2}{2}.
\end{equation}
In particular, if $u \in M_{s, 1}^3(\R{})$ has $\langle u, P\rangle \geq 0$ for all $P \in \mathcal P$, then
$$
\frac{(1-\abs{u}^2)^2}{3} \leq \abs{u-u^2}^2 \leq \frac{(1-\abs{u}^2)^2}{2}.
$$
Under these conditions, if ${\rm dist}(u, {\mathcal P}) \leq \delta \leq \frac{1}{4}$, it is not hard to check that
$$
\frac{2}{3}{\rm dist}(u, {\mathcal P}) \leq 2\frac{\abs{u-u^2}}{1-\delta} \leq \frac{1-\abs{u}^2}{1-\delta}.
$$
Finally, again for $u \in M_{s, 1}^3(\R{})$ such that $\langle u, P\rangle \geq 0$ for all $P \in \mathcal P$, a lengthy, but ultimately straight forward minimization shows that
$$
(1-\abs{u}^2)^2 \geq 12 \,\, {\rm det}(u).
$$
Hence, for $1\leq \beta < 3$, and $u \in M_{s, 1}^3(\R{})$ such that $\langle u, P\rangle \geq 0$ for all $P \in \mathcal P$, we have
\begin{align*}
W_\beta(u) &= \frac{1}{4}(1-\abs{u}^2)^2 - \beta {\rm det}(u) = \frac{3-\beta}{12} ( 1 - \abs{u}^2)^2 + \beta ( \frac{1}{12}(1-\abs{u}^2)^2 - {\rm det}(u) )\\  &\geq \frac{3-\beta}{12}(1-\abs{u}^2)^2 \geq \frac{(3-\beta)}{6}({\rm dist}(u, {\mathcal P}))^2.
\end{align*}

\begin{remark}
There is a small but confusing error in \cite{GM}.  There, the potential $W_{\beta}$ is written as
$$
W_\beta(u) = \frac{1}{2}(1-\abs{u}^2)^2 - \beta \, {\rm det}(u).
$$
As can be seen from \eqref{further_Cayley_Hamilton}, in order for this potential to be equal $\abs{u-u^2}^2$ we need to choose $\beta=2$.  Furthermore, in \cite{GM} we state that our results are valid for $2<\beta<6$; it should say $2\leq \beta< 6$.  In this paper, however, we use the expression given in \eqref{potential}.

\end{remark}

\begin{proposition}\label{prop_appendix}
Let $\Omega \subset \R{2}$ be a smooth, bounded, simply-connected open set, and $u_\eps : \Omega \to M_{s, 1}^3(\R{})$ be a minimizer of the LdG energy with non-contractible boundary data in $\mathcal{P}$.  Let $a\in \Omega$ be the distinguished point in $\Omega$ that \cite{GM} shows exist.  For $r>0$ such that $B_{2r}(a)\subset \Omega$, there is $\eps_0 > 0$ and a constant $C>0$ such that
$$
\abs{(\nabla u_\eps)(x)} + \frac{(1-\abs{u_\eps}^2)}{\eps^2} \leq C
$$
for all $x\in \Omega \setminus B_r(a)$, and all $0 < \eps \leq \eps_0$.
\end{proposition}

\begin{proof}
The proof follows \cite{BBH}.  Let us observe that the end of the proof of Lemma 8 of \cite{GM} shows that minimizers $u_\eps$ satisfy
$$
\limsup_{\eps\to 0} \int_{\Omega\setminus B_r(a)} \frac{W(u_\eps)}{\eps^2} = 0.
$$
We now appeal to Steps A.2 and B.2 of the proof  of Theorem 1 of \cite{BBH0}, to conclude that $W(u_\eps) \to 0$ uniformly in $\Omega \setminus B_r(a)$.  In particular, for $\delta > 0$ we can choose $\eps_0 > 0$ such that
$$
0 \leq 1-\abs{u_\eps^2}(x) \leq \delta
$$
for all $x\in \Omega \setminus B_r(a)$ and all $0 < \eps \leq \eps_0$.
\medskip
\medskip

We next recall from the appendix of \cite{GM} that
$$
\frac{4-\beta}{\eps^2}(1-\abs{u}^2) - 4W_{\frac{3\beta}{4}}  = -\Delta \frac{\abs{u}^2}{2} + \abs{\nabla u}^2 = -\langle u , \Delta u\rangle.
$$
We know from \cite{GM} that $\langle u_\eps , P \rangle \geq 0$ for all $P \in \mathcal P$, so we deduce
$$
4W_{\frac{3\beta}{4}}(u_\eps) \leq (1-\abs{u}^2)^2.
$$
Hence, we can choose $\eps_0 > 0$ small enough for
$$
\frac{4-\beta}{\eps^2}(1-\abs{u_\eps}^2) - 4W_{\frac{3\beta}{4}}(u_\eps) \geq \delta (1-\abs{u_\eps}^2)
$$
in $\Omega \setminus B_r(a)$, for all $0 <\eps \leq \eps_0$.  Since $\abs{u_\eps}\leq 1$, we conclude that
$$
\abs{\Delta u_\eps}\geq \delta (1-\abs{u_\eps}^2)
$$
in $\Omega \setminus B_r(a)$, for all $0 <\eps \leq \eps_0$.

From the Euler-Lagrange equation for $u_\eps$, we obtain
$$
-\Delta \frac{\partial u_\eps}{\partial x_k} + \frac{1}{\eps^2} (D^2_u W)(u)(\frac{\partial u_\eps}{\partial x_k}) = \frac{\partial \lambda_\eps}{\partial x_k}I,
$$
and then
$$
\Delta \abs{\nabla u_\eps}^2 = 2\abs{D^2_x u}^2 + \frac{2}{\eps^2}\sum_{k=1}^2 \langle (D^2_u W)(u)(\frac{\partial u_\eps}{\partial x_k}), \frac{\partial u_\eps}{\partial x_k} \rangle.
$$
Now, writing $v_\eps$ for the nearest element of $\mathcal P$ to $u_\eps$, we have
\begin{align*}
\langle (D^2_u W_\beta)(u_\eps)(\frac{\partial u_\eps}{\partial x_k}), \frac{\partial u_\eps}{\partial x_k} \rangle &= \langle (D^2_u W)(v_\eps)(\frac{\partial u_\eps}{\partial x_k}), \frac{\partial u_\eps}{\partial x_k} \rangle \\ &+ \langle (D^2_u W_\beta)(u_\eps)(\frac{\partial u_\eps}{\partial x_k}) - (D^2_u W_\beta)(v_\eps)(\frac{\partial u_\eps}{\partial x_k}), \frac{\partial u_\eps}{\partial x_k} \rangle.
\end{align*}
Now $v_\eps \in \mathcal P$, which is the set of minimizers of $W_\beta$, and it is easy to check that ${\rm tr}(\frac{\partial u_\eps}{\partial x_k}) = 0$.  Hence
$$
\langle (D^2_u W)(v_\eps)(\frac{\partial u_\eps}{\partial x_k}), \frac{\partial u_\eps}{\partial x_k} \rangle \geq 0.
$$
We deduce that
$$
\langle (D^2_u W_\beta)(u_\eps)(\frac{\partial u_\eps}{\partial x_k}), \frac{\partial u_\eps}{\partial x_k} \rangle \geq -C\abs{u_\eps - v_\eps} \abs{\frac{\partial u_\eps}{\partial x_k}}^2 \geq -C(1-\abs{u_\eps}^2)\abs{\frac{\partial u_\eps}{\partial x_k}}^2, 
$$
where the last inequality holds because from the comments before the proposition we have
$$
\abs{u_\eps-v_\eps} = {\rm dist}(u_\eps, {\mathcal P}) \leq C(1-\abs{u}^2).
$$
We conclude that
$$
\Delta \abs{\nabla u_\eps}^2 \geq 2\abs{D^2_x u}^2 - C\frac{(1-\abs{u_\eps}^2)}{\eps^2}\abs{\nabla u}^2 \geq 2\abs{D^2_x u}^2 - C\abs{\Delta u_\eps}\abs{\nabla u}^2.
$$
Since this implies
$$
\Delta \abs{\nabla u_\eps}^2 \geq \abs{D^2_x u_\eps}^2 - C\abs{\nabla u}^4
$$
in $\Omega\setminus B_r(a)$, we can apply Steps A.4 and B.3 of the proof of Theorem 1 of \cite{BBH0} to conclude that 
$$
\abs{\nabla u_\eps}\leq C
$$
in $\Omega \setminus B_r(a)$, for some constant independent of $\eps \in ]0, \eps_0]$.

\medskip
\medskip

Finally, we recall from \cite{GM} that
$$
\Delta \frac{\abs{u_\eps}^2}{2} = \frac{4}{\eps^2} W_{\frac{3\beta}{4}}(u_\eps) -\frac{4-\beta}{\eps^2}(1-\abs{u_\eps}^2) + \abs{\nabla u_\eps}^2.
$$
From here, $\zeta = 1-\abs{u_\eps}^2$ satisfies
$$
-\Delta \zeta + \frac{4-\beta}{\eps^2}\zeta = \abs{\nabla u_\eps}^2.
$$
Steps A.5 and B.4 of Theorem 1 of \cite{BBH0} give us the last conclusion of the proposition.

\end{proof}

\bibliographystyle{plain}
\bibliography{references}

\end{document}